\documentclass[11pt,twoside,english,leqno]{article}
\usepackage[T1]{fontenc}
\usepackage[latin9]{inputenc}
\usepackage{geometry}
\usepackage[active]{srcltx}
\usepackage{amsmath}
\usepackage{amsthm}
\usepackage{amssymb}
\usepackage{setspace}
\makeatletter
\numberwithin{equation}{section}
\numberwithin{figure}{section}
\newcommand{\lyxaddress}[1]{
\par {\raggedright #1
\vspace{1.4em}
\noindent\par}
}
\theoremstyle{plain}
\newtheorem{thm}{\protect\theoremname}[section]
  \theoremstyle{plain}
  \newtheorem{fact}[thm]{\protect\factname}
  \theoremstyle{remark}
  \newtheorem{rem}[thm]{\protect\remarkname}
  \theoremstyle{plain}
  \newtheorem{prop}[thm]{\protect\propositionname}
  \theoremstyle{plain}
  \newtheorem{lem}[thm]{\protect\lemmaname}
  \theoremstyle{remark}
  \newtheorem*{rem*}{\protect\remarkname}
  \theoremstyle{plain}
  \newtheorem{conjecture}[thm]{\protect\conjecturename}
  \theoremstyle{definition}
  \newtheorem{defn}[thm]{\protect\definitionname}

\@ifundefined{date}{}{\date{}}
\date{}
\usepackage{pdfsync}
\usepackage{ae,aecompl} 
\usepackage[toctitles]{titlesec}
\usepackage{amsfonts}
\usepackage{enumerate}
\usepackage{bbm}
\usepackage{scrextend}      
\usepackage{hyperref}
\usepackage{setspace}
\usepackage[title,titletoc]{appendix}
\let\originalleft\left
\let\originalright\right
\renewcommand{\left}{\mathopen{}\mathclose\bgroup\originalleft}
\renewcommand{\right}{\aftergroup\egroup\originalright}

\usepackage{multicol}

\hypersetup{
                    colorlinks=true,
                    linkcolor=red,
                    citecolor=blue,
                    urlcolor=green,
                    linktocpage=true,
                    pdfstartview=FitH,
                }

\makeatother

\usepackage{babel}
  \providecommand{\conjecturename}{Conjecture}
  \providecommand{\definitionname}{Definition}
  \providecommand{\factname}{Fact}
  \providecommand{\lemmaname}{Lemma}
  \providecommand{\propositionname}{Proposition}
  \providecommand{\remarkname}{Remark}
\providecommand{\theoremname}{Theorem}

\begin{document}
\global\long\def\R{\mathbb{R}}

\global\long\def\C{\mathbb{C}}

\global\long\def\bbN{\mathbb{N}}

\global\long\def\Z{\mathbb{Z}}

\global\long\def\N{\mathbb{N}}

\global\long\def\Q{\mathbb{Q}}

\global\long\def\T{\mathbb{T}}

\global\long\def\F{\mathbb{F}}

\global\long\def\Sph{\mathbb{S}}

\global\long\def\sub{\subseteq}

\global\long\def\cvx{\mbox{Cvx}\left(\R^{n}\right)}

\global\long\def\cvxo{\text{Cvx}_{0}\left(\R^{n}\right)}

\global\long\def\lcg{\text{LC}_{g}\left(\R^{n}\right)}

\global\long\def\lc{\text{LC}\left(\R^{n}\right)}

\global\long\def\dis{\text{Dis }\left(\R^{n}\right)}

\global\long\def\one{\mathbbm1}

\global\long\def\infc#1#2{#1\,+_{\text{cvx}}\,#2}

\global\long\def\supc#1#2{#1\,+_{\text{lc}}\,#2}

\global\long\def\vol{{\rm Vol}}

\global\long\def\EE{\mathbb{E}}

\global\long\def\epi#1{\text{epi}\left\{  #1\right\}  }

\global\long\def\sp{{\rm sp}}

\global\long\def\K{\mathcal{K}}

\global\long\def\A{\mathcal{A}}

\global\long\def\L{\mathcal{L}}

\global\long\def\P{\mathcal{P}}

\global\long\def\W{\mathcal{W}}

\global\long\def\iprod#1#2{\langle#1,\,#2\rangle}

\global\long\def\cvrs{{\rm Cvrs}}

\global\long\def\cvrsb{\overline{{\rm Cvrs}}}

\global\long\def\uball{B_{2}^{n}}

\global\long\def\conv{{\rm conv}}

\global\long\def\Gauss{\gamma_{n}^{\sigma}}

\global\long\def\EXP#1{\EE_{{\textstyle #1}}}

\global\long\def\eps{\varepsilon}

\global\long\def\PP{\mathbb{P}}

\global\long\def\IndUball{\one_{B_{2}^{n}}}

\global\long\def\J{{\cal J}}

\global\long\def\N{{\cal N}}

\global\long\def\inte#1{{\rm int}\left(#1\right)}

\global\long\def\co{C_{0}\left(\R^{n}\right)}

\global\long\def\mm{\mathcal{M}}

\global\long\def\MM{\widetilde{{\cal M}}_{n}}

\global\long\def\CC{C_{b}}

\global\long\def\cl{C_{l}}

\global\long\def\co{C_{0}\left(\R^{n}\right)}

\global\long\def\cop{C_{0}^{+}\left(\R^{n}\right)}

\global\long\def\fc{{\cal C}\left(\R^{n}\right)}

\global\long\def\norm#1{\left\Vert #1\right\Vert }

\global\long\def\supp#1{{\rm supp}\left(#1\right)}

\global\long\def\gr{g_{-}}

\global\long\def\refl#1{#1_{-}}

\global\long\def\bary#1{{\rm bary}\left(#1\right)}

\title{Functional Covering Numbers}

\author{Shiri Artstein-Avidan and Boaz A. Slomka}
\maketitle
\begin{abstract}
We define covering and separation numbers for functions. We investigate
their properties, and show that for some classes of functions there
is exact equality of separation and covering. We provide analogues
for various geometric inequalities on covering numbers, such as volume
bounds, bounds connected with Hadwiger's conjecture, and inequalities
about $M$-positions for geometric log-concave functions. In particular, we obtain strong versions of $M$-positions for geometric log-concave
functions. 
\end{abstract}

\lyxaddress{\textbf{Keywords}: Covering numbers, functionalization of geometry,
log-concave functions, duality, volume bounds, $M$-position.}

\section{Introduction}

\subsection{Background and Motivation}

Covering numbers can be found in various fields of mathematics, including
combinatorics, probability, analysis and geometry. They participate
in the solution of many problems in a natural manner, see the book
\cite[Chapter 4]{AGM15} and references therein. Loosely speaking,
their use can be seen as follows: When working with a set, or a body,
and considering some monotone property of it (such as volume, say)
one can sometimes replace the original body with the union of simpler
bodies (say balls), to obtain bounds on the needed quantity. To this
end, one computes the least number of balls of a certain radius needed
to cover the original set, this is called a covering number, see \eqref{eq:Cov_Def}
below for the formal definition.

The fact that geometric notions and inequalities have analytic counterparts
is considered folklore in the theory of asymptotic convex geometry.
This fashion of ``functionalization\textquotedbl{} started in the
90s and has proven to be very fruitful, see \cite{Milman08}. Since
covering numbers play a considerable part in the theory of convex
geometry, their extension to the realm of log-concave functions is
an essential building block for this theory.

The first step towards this end was given in \cite{ArtsteinRaz2011}
and in \cite{ArtSlom14-FracCov}, where the weighted notions of covering
and separation numbers of convex bodies were introduced and the relations
between these and the classical notions of covering and separation
were investigated.

In this note we define functional covering and separation numbers,
and discuss in detail their basic properties. We then show duality
between the two notions, which is a nontrivial example of infinite
dimensional linear programming duality. In the second part of the
note we discuss some more advanced results on functional covering
numbers. These include volume bounds of various types, geometric duality
results in the form of König and Milman, and results regarding the
$M$-position of functions. Sudakov-type estimates for functional
covering numbers will appear in \cite{Slo17}. We consider the notions
introduced here to be both novel and natural, and believe that they
will soon become an innate part of the theory of Asymptotic Geometric
Analysis. 

\subsection{Definitions}

\subsubsection{Functional covering numbers}

Given three measurable functions $f,g,h:\R^{n}\to[0,\infty)$ we define
the $h$-covering number of $f$ by $g$ 
\[
N(f,g,h)=N^{h}(f,g)=\inf\{\int hd\mu:\mu*g\ge f\}.
\]
The infimum is taken over all non-negative Borel measures $\mu$ on
$\R^{n}$. Each $\mu$ which satisfies $\mu*g\ge f$, that is, 
\[
\int g(x-t)d\mu(t)\ge f(x)\qquad{\rm for~all~}x\in\R^{n},
\]
is called a ``covering measure\textquotedbl{} of $f$ by $g$. In
the case of $h\equiv1$ we thus infimize the total mass of a covering
measure of $f$ by $g$. For general $h$ we infimize a different
quantity, namely the integral of $h$ with respect to $\mu$. It is
useful to note that the choice of $h$ does not influence the set
of covering measures. We shall call $N^{1}(f,g)$ the functional covering
number of $f$ by $g$, and $N^{h}(f,g)$ the $h$-covering number
of $f$ by $g$.

One may define variants of this notion when the set of measures over
which one takes the infimum is chosen differently. For example, if
one allows only atomic measures of the form $\sum_{i}\delta_{x_{i}}$,
then for functions which are indicators of convex sets, and $h\equiv1$
one recovers the usual definition of covering number 
\begin{equation}
N(K,T)=\min\left\{ N\,:\,N\in\bbN,\,\,\exists x_{1},\dots x_{N}\in\R^{n};\,\,K\sub\bigcup_{i=1}^{N}\left(x_{i}+T\right)\right\} .\label{eq:Cov_Def}
\end{equation}
If one lets $h$ equal to $1$ on $K$ and $+\infty$ outside of $K$,
one recovers $\overline{N}(K,T)$, that is, the classical covering
number variant when the cover centers are forced to lie inside $K$.
Another natural set of measures to discuss is that of discrete measures,
namely weighted sums $\sum_{i}w_{i}\delta_{x_{i}}$ where $w_{i}\ge0$.
These were the ones considered in \cite{ArtsteinRaz2011} and discussed
in \cite{ArtSlom14-FracCov}, again for weight function $h=1$.

\subsubsection{Functional separation numbers}

Similarly, we extend the notion of separation numbers, which is a
dual notion to that of covering, to the functional setting. Given
three measurable functions $f,g,h:\R^{n}\to[0,\infty)$ we define
the $h$-separation number of $f$ by $g$ 
\[
M(f,g,h)=M^{h}(f,g)=\sup\{\int fd\rho:\rho*g\le h\}.
\]
The supremum is taken over all non-negative Borel measures $\rho$
on $\R^{n}$. Each $\rho$ which satisfies $\rho*g\le h$, that is,
\[
\int g(x-t)d\rho(t)\le h(x)\qquad{\rm for~all~}x\in\R^{n},
\]
is called a ``separation measure\textquotedbl{} of $g$ with respect
to $h$. An interesting case here is when $f=1_{K}$ is the indicator
of some (say, convex) set and then one supremizes the total weight
of a separation measure (of $g$ with respect to $h)$ which is supported
on $K$. When $h=1_{K}$ then no mass is allowed outside of $K$ and
this corresponds to the notion of ``packing''. We shall call $M^{1}(f,g)$
the functional separation number of $f$ by $g$, and $M^{h}(f,g)$
the $h$-separation number of $f$ by $g$.

Again one may define variants of this notion when the set of measures
over which one takes the supremum is chosen differently. For example,
if one allows only atomic measures of the form $\sum_{i}\delta_{x_{i}}$,
then for functions which are indicators of convex sets, and $h\equiv1$
one recovers the usual definition of separation number 
\[
M(K,T)=\max\left\{ M\,:\,N\in\bbN,\,\,\exists x_{1},\dots x_{M}\in K\,;\,\,\left(x_{i}+T\right)\cap\left(x_{j}+T\right)=\emptyset\,\,\forall i\neq j\right\} .
\]

\subsection{Main Results \label{sec:MainResults}}

\subsubsection{Duality between covering and separation}

As in the case of convex bodies and classical theory, covering and
separation numbers are intimately related. In fact, the relation is
more exact in the functional setting, and our first main result is
an equality between the two, under certain conditions on the functions
involved. Define $u_{-}(x)=u(-x)$ for a function $u:\R^{n}\to\R$.

The inequality $M^{h}\left(f,\gr\right)\le N^{h}\left(f,g\right)$
is particularly simple, and is valid for any three measurable functions
$f,g,h:\R^{n}\to[0,\infty)$, see Proposition \ref{prop:WeakDuality}
below. In the language of linear programming, this is called ``weak
duality\textquotedbl{}. When there is equality in this inequality,
we say there is ``strong duality\textquotedbl{}, adopting the language
of linear programming. Our first main result is a strong duality between
functional covering and separation numbers under certain conditions
on the functions. Some of these conditions can later be removed. Removing
these conditions is usually quite technical. Our first result is concerns
the space $\co$ of continuous real valued functions on $\R^{n}$
which vanish at infinity.
\begin{thm}
\label{thm:StrongDuality-Co-1}Let $0\neq f,g,h\in\co$. Suppose that
$f$ is compactly supported. Then 
\[
M^{h}\left(f,\gr\right)=N^{h}\left(f,g\right).
\]
Moreover, there exists a $g_{-}$-separated measure $\rho$ such that
$\int fd\rho=M^{h}\left(f,g_{-}\right).$ 
\end{thm}
Theorem \ref{thm:StrongDuality-Co-1} follows from the fact that the
numbers $N^{h}(f,g)$ and $M^{h}(f,g_{-})$ can be interpreted as
the outcomes of two dual problems in the sense of linear programming,
and is a direct consequence of \cite[Theorem 7.2]{Barvinok2002},
a zero gap result for linear programming duality in a very general
setting of ordered topological vector spaces.

\noindent The case where $h\equiv1$ (which is not in $C_{0}$) is
of particular importance, and we establish a strong duality relation
in this case as well:
\begin{thm}
\noindent \label{thm:StrongDuality}Let $0\neq f,g\in\co$ and assume
that there exists a finite regular Borel measure $\mu$ which covers
$f$ by $g$. Then 
\[
M^{1}(f,g_{-})=N^{1}(f,g).
\]
Moreover, there exists a covering measure $\mu_{0}$ of $f$ by $g$,
such that $\mu_{0}(\R^{n})=N^{1}\left(f,g\right)$. 
\end{thm}
\noindent The proof of Theorem \ref{thm:StrongDuality} is based on
a variation of \cite[Theorem 7.2]{Barvinok2002}. For the convenience
of the reader, we state and prove a single linear programming duality
result from which both Theorems \ref{thm:StrongDuality-Co-1} and
\ref{thm:StrongDuality} follow. This result is given as Theorem \ref{thm:SD_cov_gen}
in Section \ref{sec:StrongDuality}. We also prove the following two
extensions of Theorems \ref{thm:StrongDuality-Co-1} and \ref{thm:StrongDuality},
obtained via limiting arguments.
\begin{thm}
\label{thm:StrongDuality-Co}Let $0\neq f,g,h\in\co$. Suppose that
$\int f,\int g<\infty$, and $N^{h}\left(f,g\right)<\infty$. Then
\[
M^{h}\left(f,\gr\right)=N^{h}\left(f,g\right).
\]
\end{thm}
\begin{thm}
\label{thm:SD_ext_1}Let $f,g:\R^{n}\to\R^{+}$ be measurable. Suppose
that $\left(g_{k}\right)\sub\co$ is a non-increasing sequence converging
point-wise to $g$, and that $N^{1}\left(f,g\right)<\infty$. Then
$M^{1}\left(f,\gr\right)=N^{1}\left(f,g\right)=\lim N^{1}\left(f,g_{k}\right)$.
Moreover, there exists a covering measure $\mu$ of $f$ by $g$ such
that $\mu\left(\R^{n}\right)=N^{1}\left(f,g\right)$. 
\end{thm}
Finally, in the case most relevant for convex geometry, namely that
of $h=1$ and where $f$ and $g$ are geometric log-concave functions,
we have again a strong duality result. More precisely, let $LC_{g}(\R^{n})$
denote the class of functions $f:\R^{n}\to[0,1]$ which are upper
semi continuous, $-\log f$ is convex, and $f(0)=0$. These are called
geometric log-concave functions and play a central role in convex
geometry and its functional extensions. The following theorem holds,
and its proof will appear in \cite{Slo17}. 
\begin{thm}
\label{thm:Sd-lcf}Let $f,g\in\lcg$. Then $M(f,g_{-})=N(f,g)$. 
\end{thm}

\subsubsection{Volume estimates}

A main tool in estimating classical covering numbers are so called
``volume bounds'', where the covering numbers are bounded from above
and from below by ratios of volumes. We provide two such bounds, for
geometric log-concave functions, where volume is replaced by integral.
We use $*$ to denote usual convolution as above, and use $\star$
to denote the sup-convolution operation, defined by $(f\star g)(x)=\sup_{z}f(z)g(x-z)$
(and sometimes playing the role of Minkowski addition in the functionalization
of convex geometry). We show
\begin{thm}
Let $f,g\in LC_{g}(\R^{n})$, then 
\[
\frac{\int f^{2}\left(x\right)dx}{\|f*g_{-}\|_{\infty}}\le N\left(f,g\right)\le2^{n}\frac{\int f^{2}\left(x\right)dx}{\|f*g_{-}\|_{\infty}}.
\]
and for every $p>1$
\[
\frac{\int f\left(x\right)dx}{\int g\left(x\right)dx}\le N\left(f,g\right)\le\frac{\int\left(f\star g_{-}^{p-1}\right)\left(x\right)dx}{\int g_{-}^{p}\left(x\right)dx},
\]
\end{thm}

\subsubsection{Functional $M$-position}

We provide a covering-number definition for the $M$-position of a
convex body, and show that it is equivalent to a volume-type definition
in the spirit of Klartag and Milman \cite{KM05}. We show that there
exists a universal constant $C>0$ such that every geometric log-concave
function has an $M$-position with constant $C$, and as a result
get some extensions of the functional reverse Brunn-Minkowski inequality
of Klartag and Milman, in particular to the non-even case. Denoting
by $g_{0}:\R^{n}\to(0,1]$ the gaussian $g_{0}(x)=\exp(-|x|^{2}/2)$
we show
\begin{thm}
\label{thm:MpositionMAIN}There exists a universal constant $C>0$,
and for any $n$ and any function $f\in LC_{g}(\R^{n})$ there exists
$T_{f}\in GL_{n}$, such that denoting $\tilde{f}=f\circ T_{f}$ we
have that $\int f=(2\pi)^{n/2}$ and the following properties hold:
\[
\max\{N(\tilde{f},g_{0}),N(\tilde{f}^{*},g_{0}),N(g_{0},\tilde{f}),N(g_{0},\tilde{f}^{*})\}\le C^{n}
\]
and, for every $h\in LC_{g}(\R^{n})$
\[
\frac{1}{C^{n}}\int g_{0}\star h\le\int f\star h\le C^{n}\int g_{0}\star h
\]
 and 
\[
\frac{1}{C^{n}}\int g_{0}^{*}\star h\le\int f^{*}\star h\le C^{n}\int g_{0}^{*}\star h.
\]
\end{thm}
Here we denoted for $f=\exp(-\varphi)$ its log-Legendre dual by $f^{*}=\exp(-{\cal L}\varphi)$
where ${\cal L}\varphi(y)=\sup\left(\iprod yx-\varphi(x)\right)$
is the Legendre transform. As a tool in the proof of this theorem,
but also of independent interest, we give a König-Milman \cite{KonigMilman87}
type result connecting the covering of $f$ by $g$ and the covering
of $g^{*}$ by $f^{*}$ which are their log-Legendre duals. We show
that there exists a universal $C>0$ (independent of dimension) such
that for any $n$ and any $f,g\in LC_{g}(\R^{n})$ we have 
\[
C^{-n}N\left(g^{*},f^{*}\right)\le N\left(f,g\right)\le C^{n}N\left(g^{*},f^{*}\right).
\]

The paper is organized as follows. In Section \ref{sec:Basic-identities-and}
we gather the basic identities and simple inequalities for functional
covering numbers, both for use in this paper and a for future reference.
In Section \ref{sec:StrongDuality} we prove Theorems \ref{thm:StrongDuality-Co-1}
up to \ref{thm:SD_ext_1}. To this end we start with a weak duality
result, then prove an infinite dimensional linear programing duality
result, which serves as the main ingredient in the proofs. In Section
\ref{sec:Volume-bounds} we prove the volume bounds described above.
In Section \ref{sec:Functional-Hadwiger-conjecture} we discuss Hadwiger's
conjecture, we show it is valid in the functional setting for even
functions, and provide some bound for the general case. Finally, in
Section \ref{sec:M-position-for} we define functional $M$-position
in two different ways, one via volume and the other via covering,
and show that they are equivalent. We then prove a König-Milman type
geometric duality result, connecting the covering of $f$ by $g$
and that of their Legendre duals. Finally, we give two proofs that
every centered geometric log-concave function admits a functional
$M$-position with a universal constant $C>0$. One proof using the
functional reverse Brunn-Minkowski inequality of Klartag and Milman,
and the other following directly from the geometric theorem of Milman
on the existence of $M$-positions for bodies. 

\subsection*{Acknowledgments}

The first named author was supported by ISF grant number 665/15.

\section{Basic identities and inequalities\label{sec:Basic-identities-and}}

Since this note is the first time the functional covering numbers
$N(f,g,h)=N^{h}(f,g)$ and the functional separation numbers $M(f,g,h)=M^{h}(f,g)$
are introduced, we devote a section to pointing out some of the useful
properties of these numbers. The proofs for most of the facts below
follow directly from the definitions and are thus omitted. We leave
only the ones which are slightly less self-evident. 

\subsubsection*{Linear transformations}
\begin{fact}
Define $u_{a}(x)=u(x-a)$ for $u:\R^{n}\to[0,\infty)$ and $a\in\R^{n}$,
then for measurable functions $f,g,h:\R^{n}\to[0,\infty)$ one has
that 
\[
N(f,g,h)=N(f_{a},g_{a},h)=N(f,g_{a},h_{-a})=N(f_{a},g,h_{a}).
\]
\end{fact}
\begin{fact}
\noindent \label{fact:Cov_Under_Lin}Define $u_{A}=u(Ax)$ for $A\in GL_{n}$
and $u:\R^{n}\to[0,\infty)$, then for measurable functions $f,g,h:\R^{n}\to[0,\infty)$
one has that 
\[
N(f,g,h)=N(f_{A},g_{A},h_{A}).
\]
\end{fact}
\begin{fact}
For measurable functions $f,g,h:\R^{n}\to[0,\infty)$ and positive
constants $a,b,c>0$ one has that 
\begin{eqnarray*}
N(af,bg,ch) & = & \frac{ac}{b}N(f,g,h).
\end{eqnarray*}
\end{fact}

\subsubsection*{Sub-additivity }
\begin{fact}
For measurable functions $f_{1},f_{2},g,h:\R^{n}\to[0,\infty)$ one
has $N\left(f_{1}+f_{2},g,h\right)\le N\left(f_{1},g,h\right)+N\left(f_{2},g,h\right)$. 
\end{fact}
\begin{fact}
For measurable functions $f,g,h_{1},h_{2}:\R^{n}\to[0,\infty)$ one
has $N\left(f,g,h_{1}+h_{2}\right)\le N\left(f,g,h_{1}\right)+N\left(f,g,h_{2}\right).$
\end{fact}

\subsubsection*{Monotonicity}
\begin{fact}
\label{fact:cov_mono}For measurable functions $f_{1},g_{1},h_{1},f_{2},g_{2},h_{2}:\R^{n}\to[0,\infty)$
such that $f_{1}\le f_{2}$, $g_{1}\ge g_{2}$ and $h_{1}\le h_{2}$
one has that 
\[
N(f_{1},g_{1},h_{1})\le N(f_{2},g_{2},h_{2})\qquad{\rm and}\qquad M(f_{1},g_{1},h_{1})\le M(f_{2},g_{2},h_{2}).
\]
\end{fact}

\subsubsection*{Convolutions}

Two types of convolutions are often used for log-concave functions
(in fact, there are more, but we restrict to these two for simplicity
of the exposition). The first is the standard convolution of $L_{1}$
functions given by 
\[
(f*g)(x)=\int f(t)g(x-t)dt
\]
which can be defined also for measures by 
\[
\int f\,d(\mu*\nu)=\int f(x+t)d\mu(x)d\nu(t).
\]
This convolution is a very standard operation in analysis. It follows
from functional Brunn-Minkowski theory, the Prekopa-Leindler inequality,
that the convolution of two log-concave functions is again log-concave.

The second type of convolution we shall need is the so-called sup-convolution
or Asplund product, given by 
\[
(f\star g)(x)=\sup_{z}f(z)g(x-z).
\]
This operation is sometimes considered in Asymptotic Geometric Analysis
as a functional analogue of Minkowski addition for convex bodies.
For an account of which operation should be considered as the ``natural''
analogue of Minkowski addition the reader is referred to \cite{Milman08}
and the many references therein.

Let us describe the monotonicity properties of covering numbers with
respect to such convolutions. 
\begin{fact}
Let $f,g,h,\varphi:\R^{n}\to[0,\infty)$ be measurable then 
\[
N(f,g,h)\ge N(f*\varphi,g*\varphi,h)\qquad{\rm and}\qquad M(f,g,h)\le M(f,g*\varphi,h*\varphi)
\]
\end{fact}
\begin{proof}
Indeed, if $\mu$ is a covering measure of $f$ by $g$, then for
any non-negative $\varphi$ we have that $\mu$ covers $\varphi*f$
by $\varphi*g$. So, we are infimizing the same linear function on
a larger set as there may be other covering measures of $\varphi*f$
by $\varphi*g$. For the second inequality, note that any $\rho$
for which $\rho*g\le h$, will also satisfy that $\rho*g*\varphi\le h*\varphi$,
so we are supremizing the same linear functional on a larger set (as
there may be other separation measures of $g*\varphi$ by $h*\varphi$).
Thus the supremum of the latter is greater than or equal to the former. 
\end{proof}
\begin{fact}
Let $f,g,h,\varphi:\R^{n}\to[0,\infty)$ be measurable then 
\[
N(f,g,h*\varphi_{-})\ge N(f*\varphi,g,h)\qquad{\rm and}\qquad M(f*\varphi_{-},g,h)\le M(f,g,h*\varphi)
\]
\end{fact}
\begin{proof}
Indeed, if $\mu$ is a covering measure of $f$ by $g$, that is,
$\mu*g\ge f$, then for any non-negative $\varphi$ we have that $\mu*\varphi$
covers $f*\varphi$ by $g$. So, when computing $N(f*\varphi,g,h)$
we are infimizing over a set which contains $\mu*\varphi$ for any
$\mu$ which is a covering measure of $f$ by $g$. In particular,
this infimum will be less that or equal to the following number, whenever
$\mu$ is a covering measure of $f$ by $g$: 
\[
\int hd(\mu*\varphi)=\int h*\varphi_{-}d\mu
\]
Therefore, if we choose to infimize the linear functional coming from
$h*\varphi_{-}$ over all covering measures of $f$ by $g$, we shall
get a greater (than or equal to) result than when we infimize integration
with respect to $h$ on the set of all covering measure of $f*\varphi$
by $g$.

Similarly, note that any $\rho$ for which $\rho*g\le h$, will also
satisfy that $\rho*\varphi*g\le h*\varphi$, so that $\rho*\varphi$
is a separation measure of $g$ with respect to $h*\varphi$ whenever
$\rho$ is a separation measure of $g$ with respect to $h$. When
we compute $\int fd(\rho*\varphi)=\int f*\varphi_{-}d\rho$ and take
supremum over all $\rho$ which are $g$ separation measures with
respect to $h$ we are going to get a smaller (than or equal to) result,
since there may be more $g$ separation measures with respect to $h*\varphi$,
not coming from $g$ separation measures with respect to $h$ which
were convolved with $\varphi$. 
\end{proof}
Next, we give a similar monotonicity result with respect to sup-convolution,
analogous to the inequality $N(A,B)\ge N(A+C,B+C)$ for classical
covering numbers.
\begin{fact}
\label{fact:two-convs}Let $f,g,h,\varphi:\R^{n}\to[0,\infty)$ be
measurable then $N(\varphi\star f,\varphi\star g,h)\le N(f,g,h)$. 
\end{fact}
\begin{proof}
We shall use the easily verified fact that for any three functions
\begin{equation}
f_{1}*(f_{2}\star f_{3})\ge f_{2}\star(f_{1}*f_{3})\label{eq:Sup-conv_Conv_ineq}
\end{equation}
(and the corresponding fact for measures). Indeed, 
\[
\left(\mu*\left(\varphi\star g\right)\right)\left(x\right)=\int\sup_{z}\varphi\left(z\right)g\left(x-y-z\right)d\mu\left(y\right)\ge\sup_{z}\varphi\left(z\right)\int g\left(x-z-y\right)d\mu\left(y\right)=\varphi\star(\mu*g)\left(x\right).
\]
Therefore if $\mu$ is a covering measure of $f$ by $g$ (that is,
$\mu*g\ge f$) then $\mu$ is also a covering measure of $\varphi\star f$
by $\varphi\star g$, from which this fact follows. 
\end{proof}

\subsubsection*{Sub-Multiplicativity}

The next few results require an additional assumption on the weight
functions $h$ associated with the covering number $N\left(f,g,h\right)$.
We will assume that $h\left(x+y\right)\le h_{1}\left(x\right)h_{2}\left(y\right)$
for the measurable functions $h,h_{1},h_{2}$ used. The log-sub-additive
case of a single weight function $h_{1}=h_{2}=h$ such that $h$ satisfies
$h(x+y)\le h(x)h(y)$ is of particular interest, and in particular
the case where $h\equiv1$ is included.

The following inequality is an analogue of $N\left(A,B\right)\le N\left(A,C\right)N\left(C,B\right)$
for convex bodies. 
\begin{fact}
\label{fact:sub_mult}Let $f,g,\varphi,h,h_{1},h_{2}$ be measurable
and assume that $h\left(x+y\right)\le h_{1}\left(x\right)h_{2}\left(y\right)$
for all $x,y\in\R^{n}$. Then 
\[
N(f,g,h)\leq N(f,\varphi,h_{1})N(\varphi,g,h_{2}).
\]
\end{fact}
\begin{proof}
Indeed, if $\mu$ is a covering measure of $f$ by $\varphi$ and
$\nu$ is a covering measure of $\varphi$ by $g$ then 
\[
\mu*\nu*g\ge\mu*\varphi\ge f,
\]
and 
\[
\int hd\left(\mu*\nu\right)=\int\int h\left(x+y\right)d\mu\left(x\right)d\nu\left(y\right)\le\int h_{1}\left(x\right)d\mu\left(x\right)\int h_{2}\left(y\right)d\nu\left(y\right).
\]
By infimizing over all covering measures $\mu$and $\nu$ we get $N(f,g,h)\leq N(f,\varphi,h_{1})\,N(\varphi,g,h_{2})$. 
\end{proof}
The next result is a functional analogue for $N\left(A+B,C+D\right)\le N\left(A,C\right)N\left(B,D\right)$.
\begin{fact}
\label{fact:Addition_two_fcns}Let $f,g,\varphi,\psi,h,h_{1},h_{2}$
be measurable and assume that $h\left(x+y\right)\le h_{1}\left(x\right)h_{2}\left(y\right)$
for all $x,y\in\R^{n}$. Then 
\[
N(f\star\varphi,g\star\psi,h)\le N(f,g,h_{1})N(\varphi,\psi,h_{2})
\]
\end{fact}
\begin{rem}
Note that Fact \ref{fact:Addition_two_fcns} implies Fact \ref{fact:two-convs}
under the assumptions that $h\left(x+y\right)\le h\left(x\right)h\left(y\right)$
and $h\left(0\right)=1$. Indeed, $\mu=\delta_{0}$ is a covering
measure of $\varphi$ by $\varphi$ and hence $N(\varphi,\varphi,h)\le h(0)=1$,
from which it follows that 
\[
N(f\star\varphi,g\star\varphi,h)\le N(f,g,h)N(\varphi,\varphi,h)\le N(f,g,h).
\]
Note that if $\min_{x}h\left(x\right)=h\left(0\right)$ one actually
has $N\left(\varphi,\varphi,h\right)=h\left(0\right)$ as for any
covering measure $\mu$ of $\varphi$ by itself, $\int h\,d\mu\ge h(0)\,\int d\mu,$
and $\int d\mu\ge1$ which follows by integrating the inequality $\varphi*\mu\ge\varphi$. 
\end{rem}
\begin{proof}
[Proof of Fact \ref{fact:Addition_two_fcns}] Indeed, as in the proof
of Fact \ref{fact:two-convs}, we know that if $\mu$ is a covering
measure for $f$ by $g$ then $g*\mu\ge f$ so that also 
\[
(\psi\star g)*\mu\ge\psi\star(g*\mu)\ge\psi\star f.
\]
Hence $\mu$ is a covering measure of $\psi\star f$ by $\psi\star g$.
If $\nu$ is a covering measure of $\varphi$ by $\psi$, then $\nu*\psi\ge\varphi$,
and hence 
\[
(\psi\star g)*\mu*\nu\ge(\psi\star f)*\nu\ge f\star(\psi*\nu)\ge f\star\varphi
\]
so that $\mu*\nu$ is a covering measure of $f\star\varphi$ by $g\star\psi$.
Therefore 
\begin{align*}
N(f\star\varphi,g\star\psi,h) & \le\int hd(\mu*\nu)=\int h(s+t)d\mu(s)d\nu(t)\\
\le & \int h_{1}d\mu\int h_{2}d\nu=N(f,g,h_{1})N(\varphi,\psi,h_{2}).
\end{align*}
\end{proof}
Interestingly, a similar result holds when sup-convolution is replaced
by usual convolution: 
\begin{fact}
Let $f,g,\varphi,\psi,h,h_{1},h_{2}$ be measurable and assume that
$h\left(x+y\right)\le h_{1}\left(x\right)h_{2}\left(y\right)$ for
all $x,y\in\R^{n}$. Then 
\[
N\left(\varphi*f,\psi*g,h\right)\le N(f,g,h_{1})N(\varphi,\psi,h_{2})
\]
\end{fact}
\begin{proof}
Indeed, let $\mu$ be a covering measure of $f$ by $g$, so that
$g*\mu\ge f$ and let $\nu$ be a covering measure of $\varphi$ by
$\psi$ so that $\nu*\psi\ge\varphi$, then 
\[
g*\psi*\mu*\nu\ge f*\varphi
\]
so that $\nu*\mu$ is a covering measure of $f*\varphi$ by $g*\psi$.
Thus 
\[
\int hd\left(\mu*\nu\right)=\int\int h\left(s+t\right)d\mu\left(s\right)d\nu\left(t\right)\le\int h_{1}(s)d\mu(s)\int h_{2}(t)d\nu(t)
\]
which means that $N\left(\varphi*f,\psi*g,h\right)\le N\left(f,g,h_{1}\right)N\left(\varphi,\psi,h_{2}\right)$
as claimed.
\end{proof}

\section{\label{sec:StrongDuality}Duality between covering and separation }

In this section we show results of the form $N^{h}(f,g)=M^{h}(f,g_{-})$
for different classes of functions, and under various conditions on
$h$. We prove Theorems \ref{thm:StrongDuality-Co-1} through \ref{thm:SD_ext_1}. 

\subsection{Weak duality}

A relatively simple fact is the following. 
\begin{prop}
\label{prop:WeakDuality} Let $f,g,h\in\R^{n}\to\left[0,\infty\right)$
be measurable. Then $M^{h}\left(f,\gr\right)\le N^{h}\left(f,g\right)$. 
\end{prop}
\begin{proof}
Let $\mu$ be a covering measure of $f$ by $g$. Let $\rho$ be a
$\gr$-separated measure with respect to $h$. By our assumptions
we have that $\rho*\gr\le h$ and $\mu*g\ge f$. Thus Tonelli's theorem
implies that

\begin{align*}
\int fd\rho\le & \int\left(\mu*g\right)\left(x\right)d\rho\left(x\right)=\int d\rho\left(x\right)\int d\mu\left(y\right)g\left(x-y\right)=\int d\mu\left(y\right)\left(\rho*\gr\right)\left(y\right)\le\int hd\mu
\end{align*}
and so $M^{h}\left(f,\gr\right)\le N^{h}\left(f,g\right)$.
\end{proof}
In the sequel we shall make extensive use of the following
\begin{rem}
\label{rem:Restricted_Mes} Note that the inequality above (weak duality
relation) holds for any covering and any separating measures. Therefore,
any reverse inequality between covering and separation, even when
the infimum and supremum are taken over a {\em smaller} family
of measures, would imply equality (namely a strong duality relation)
without any restriction on the measures.
\end{rem}

\subsection{Strong duality}

In this section we prove Theorems \ref{thm:StrongDuality-Co-1}, \ref{thm:StrongDuality},
\ref{thm:StrongDuality-Co}, and \ref{thm:SD_ext_1}. The main ingredient
of the proofs is Theorem \ref{thm:SD_cov_gen}; an infinite dimensional
linear programming duality result which is a simple variation of \cite[Theorem 7.2]{Barvinok2002}.
In order to state and prove this result, we need to introduce some
notation and to recall some facts. 

We shall work with the space $\mm$ of finite countably additive regular
Borel measures on $\R^{n}$ endowed with the norm topology of total
variation. It is a well known fact that $\mm=\co^{*}$, namely it
is the space dual to $\co$ endowed with the supremum norm topology.
In the sequel, we will always assume that all covering and separation
measures in the definitions of $N\left(f,g,h\right),M\left(f,g,h\right)$
are restricted to $\mm$. This is a technical restriction under which
we will be able to establish a strong duality relation between covering
and separation numbers. By Remark \ref{rem:Restricted_Mes}, once
strong duality is established under such a restriction, it also holds
without this restriction.

There is a natural duality on $\mm\times\co$ defined by $\iprod{\mu}f=\int fd\mu$
for each $\mu\in\mm$ and $f\in\co$. For $g\in C_{0}(\R^{n})$, consider
the linear functions taking a measure $\mu\in\mm$ to the functions
$\mu*g\in\co$ and $\mu*\gr\in\co$. In the following proofs we will
use the fact that these linear functions are adjoint, namely $\iprod{\mu}{\rho*\gr}=\iprod{\rho}{\mu*g}$
for all $\mu,\rho\ge0$. Indeed, this fact follows by Tonnelli's theorem
as 
\begin{align*}
\iprod{\mu}{\rho*\gr} & =\int\int d\rho\left(y\right)\gr(x-y)d\mu(x)=\int d\rho\left(y\right)\int g\left(y-x\right)d\mu(x)=\iprod{\rho}{\mu*g}.
\end{align*}

We endow the spaces $\mm\oplus\R$ and $\co\oplus\R$ with the usual
topology of the direct sum. Fixing $g\in C_{0}(\R^{n})$ and $h:\R^{n}\to\R$
which is measurable and bounded, we define the linear transformation
$A:\mm\oplus\co\to\co\oplus\R$ by $A\left(\mu,\varphi\right)=\left(\mu*g-\varphi,\int_{\R^{n}}h\,d\mu\right)$,
and consider the image 
\[
A\left(K\right)=\left\{ \left(\mu*g-\varphi,\,\int h\,d\mu\right)\,\,:\,\,\mu\in\mm^{+},\,\varphi\in\co^{+}\right\} \sub\co\oplus\R
\]
of the positive cone $K=\mm^{+}\times\co^{+}$ . 
\begin{thm}
\label{thm:SD_cov_gen}Let $g,f\in\co$, and let $h:\R^{n}\to\R$
be a bounded continuous function. Suppose that $A\left(K\right)$
is closed, and that there exists a measure $\mu\in\mm^{+}$ such that
$\mu*g\ge f$. Then $N\left(f,g,h\right)=M\left(f,\gr,h\right)$.
Moreover, there exists an optimal covering measure $\mu_{0}\in\mm^{+}$
such that $\mu_{0}*g\ge f$ and $\int hd\mu_{0}=N\left(f,g,h\right)$. 
\end{thm}
The fact that $f,g,$ and $h$ are non-negative functions is not actually
used in the proof of Theorem \ref{thm:SD_cov_gen} (although non-negativity
is assumed in the definitions of covering and separation, one can
remove this restriction for the sake of this argument). We may therefore
apply the theorem to the functions $-f,-g$ and $-h$ instead. Note
that by definition $N\left(-f,-g,-h\right)=-M\left(h,g,f\right)$,
and $M\left(-f,-\gr,-h\right)=-N\left(h,\gr,f\right)$. We thus get: 
\begin{thm}
\label{thm:SD_sep_gen}Let $g,h\in\co$, and let $f:\R^{n}\to\R$
be a bounded continuous function. Suppose 
\[
B(K)=\left\{ \left(\rho*g+\varphi,\int f\,d\rho\right)\,\,:\,\,\rho\in\mm^{+},\,\varphi\in\co^{+}\right\} \sub\co\oplus\R
\]
is closed, and that there exists a measure $\rho_{0}\in\mm^{+}$ such
that $g*\rho_{0}\le h$. Then $N\left(f,\gr,h\right)=M\left(f,g,h\right)$.
Moreover, there exists an optimal $g$-separated measure $\rho\in\mm^{+}$
such that $\rho*g\le h$ and $\int hd\rho=M\left(f,g,h\right)$. 
\end{thm}
Before we prove Theorem \ref{thm:SD_cov_gen}, let us show how Theorems
\ref{thm:StrongDuality-Co-1}, \ref{thm:StrongDuality-Co}, and \ref{thm:StrongDuality}
follow. We begin with the proof of Theorem \ref{thm:StrongDuality-Co-1},
for which we need the following lemma. 
\begin{lem}
\label{lem:Sep_Bounded_TV}Let $0\neq g\in\co$ be non-negative, let
$h:\R^{n}\to\R^{+}$ be bounded, and let $f:\R^{n}\to\R^{+}$ be measurable
with compact support. Then there exists $C\left(f,g,h\right)>0$ such
that for any measure $\rho$ satisfying $\rho*g\le h$, there exists
a measure $\widetilde{\rho}\le\rho$ so that $\mathbf{\widetilde{\rho}}*g\le h$,
$\int f\,d\widetilde{\rho}=\int f\,d\rho$, and $\widetilde{\rho}\left(\R^{n}\right)\le C$. 
\end{lem}
\begin{proof}
Denote the support of $f$ by $K$. Since $g\neq0$ is continuous,
there exists $a>0$ and a ball $B\sub\R^{n}$ such that $g\left(x\right)\ge a$
for all $x\in B$. Since $K$ is bounded, there exist $x_{1},\dots,x_{N}\in\R^{n}$
such that $K\sub\bigcup_{i=1}^{N}\left(x_{i}+B\right)$. Thus, for
each $i$ and every measure $\rho$ which is $g$-separated with respect
to $h$, we have that for any $i\le N$
\[
\sup_{\R^{n}}h\ge\int d\rho\left(y\right)g\left(x_{i}-y\right)\ge\int_{x_{i}+B}d\rho\left(y\right)g\left(x_{i}-y\right)\ge a\rho\left(x_{i}+B\right)
\]
and so $\rho\left(x_{i}+B\right)\le\sup_{\R^{n}}h/a$ which implies
that $\rho\left(K\right)\le N\,\sup_{\R^{n}}h/a=:C$. Since $f$ is
supported in $K$, the measure $\widetilde{\rho}$ defined by $\widetilde{\rho}\left(A\right)=\rho\left(A\cap K\right)$
is a $g$-separated measure with respect to $h$, that satisfies both
$\int f\,d\widetilde{\rho}=\int f\,d\rho$ and $\widetilde{\rho}\left(\R^{n}\right)\le C$,
as required. 
\end{proof}
\begin{proof}
[Proof of Theorem \ref{thm:StrongDuality-Co-1}] In order to invoke
Theorem \ref{thm:SD_sep_gen}, we need to show that two conditions
are satisfied. The first is the existence of a a measure $\rho_{0}$
such that $\rho_{0}*g\le h$, for which we may simply take $\rho_{0}\equiv0$.
The second condition is that 
\[
B\left(K\right)=\left\{ \left(\rho*g+\varphi,\,\int f\,d\rho\right)\,\,:\,\,\rho\in\mm^{+},\,\varphi\in\co^{+}\right\} 
\]
is closed. Indeed, take a sequence $\rho_{k}*g+\varphi_{k}$ which
converges to $\psi\in\co^{+}$ and $\iprod{\rho_{k}}f\to\alpha\in\R^{+}$.
For sufficiently large $k$, we have that $\rho_{k}*g\le\psi+1$.
Hence, by Lemma \ref{lem:Sep_Bounded_TV}, there exists a uniformly
bounded sequence $\left(\widetilde{\rho}_{k}\right)$ such that $\widetilde{\rho}_{k}\le\rho_{k}$
and $\int f\,d\widetilde{\rho}_{k}=\int f\,d\rho_{k}$. By the Banach-Alaoglu
theorem, we may assume without loss of generality that $\widetilde{\rho}_{k}$
converges in the weak{*} topology to some measure $\widetilde{\rho}\in\mm^{+}$.
In particular, since $g$ is continuous, $\widetilde{\rho}_{k}*g\to\widetilde{\rho}*g\in\co$
point-wise. Since $\widetilde{\rho}_{k}\le\rho_{k}$, it follows that
\[
\widetilde{\rho}*g\leftarrow\widetilde{\rho}_{k}*g\le\widetilde{\rho}_{k}*g+\varphi_{k}\le\rho_{k}*g+\varphi_{k}\to\psi
\]
Hence, $B\left(\widetilde{\rho},\psi-\widetilde{\rho}*g\right)=\left(\psi,\alpha\right)\in B\left(K\right)$,
which means that $B\left(K\right)$ is closed. 
\end{proof}
Next we prove Theorem \ref{thm:StrongDuality}. 
\begin{proof}
[Proof of Theorem \ref{thm:StrongDuality}] Here $h=1$, we have $f,g\in C_{0}(\R^{n})$,
and we would like to show that the conditions of Theorem \ref{thm:SD_cov_gen}
hold. Suppose that $A\left(K\right)\ni\left(\psi_{k},\alpha_{k}\right)\to\left(\psi,\alpha\right)\in\left(\co^{+},\R^{+}\right)$.
This means that there exists a sequence $\left(\mu_{k}\right)$ in
$\mm^{+}$ and a sequence $\left(\varphi_{k}\right)$ in $\co^{+}$
such that $g*\mu_{k}-\varphi_{k}\to\psi$ and $\mu_{k}\left(\R^{n}\right)\,\to\alpha$.
By the Banach-Alaoglu theorem we may assume without loss of generality
that $\mu_{k}$ converges in the weak{*} topology to some measure
$\mu\in\mm^{+}$. In particular, since $g$ is continuous, $\mu_{k}*g\to\mu*g$
point-wise and so $\varphi_{k}$ converges to some continuous function
$\varphi\ge0$. Clearly, we also have $\mu\left(\R^{n}\right)\le\alpha$.
If $\mu\left(\R^{n}\right)=\alpha$, then $\left(\psi,\alpha\right)\in A\left(K\right)$
as needed. Suppose that $\mu\left(\R^{n}\right)<\alpha$. The case
$\mu\left(\R^{n}\right)=0$ cannot occur as $\mu*g\ge f\neq0$, hence
$c\cdot\mu\left(\R^{n}\right)=\alpha$ for some $c>1$. The measure
$\widetilde{\mu}$ defined by $\widetilde{\mu}\left(B\right)=c\cdot\mu\left(B\right)$,
and the function $\widetilde{\varphi}=\varphi+\left(c-1\right)\mu*g\in\co^{+}$,
thus satisfy that $\tilde{\mu}*g-\widetilde{\varphi}=\psi$ and $\widetilde{\mu}\left(\R^{n}\right)=\alpha$,
which means that $\left(\psi,\alpha\right)\in A\left(K\right)$ and
$A\left(K\right)$ is closed.

Since $N\left(f,g\right)<\infty$ means that there exist some covering
measure $\mu\in\mm^{+}$ of $f$ by $g$, we may apply Theorem \ref{thm:SD_cov_gen}
to complete the proof. 
\end{proof}
Next, we prove Theorems \ref{thm:StrongDuality-Co} and \ref{thm:SD_ext_1}. 
\begin{proof}
[Proof of Theorem \ref{thm:StrongDuality-Co}] Let $\left(f_{k}\right)$
be a non-decreasing sequence of compactly supported functions in $\co$,
which converges to $f$ in norm, that is $\sup_{x}\left|f\left(x\right)-f_{n}\left(x\right)\right|\to0$.
By Proposition \ref{prop:WeakDuality}, 
\[
M\left(f,g_{-},h\right)\le N\left(f,g,h\right).
\]

By Theorem \ref{thm:StrongDuality-Co-1} we have that $M\left(f_{k},g_{-},h\right)=N\left(f_{k},g,h\right)$.
Moreover, we clearly have that 
\[
M\left(f,g_{-},h\right)\ge M\left(f_{k},g_{-},h\right)=N\left(f_{k},g,h\right),
\]
and therefore it is sufficient to show that $\lim_{k}N\left(f_{k},g,h\right)\ge N\left(f,g,h\right)$.
Indeed, $N\left(f_{k},g,h\right)\le N\left(f,g,h\right)$ is a monotonically
increasing function which has a limit. Assume that there exists $\eps>0$
such that $\lim_{k}N\left(f_{k},g,h\right)=N\left(f,g,h\right)-\eps$.
Let $c>0$ so that $c\int hdx=\eps/2$, and let $\delta=c\int gdx$.
Fix $k$ large enough so that $\sup\left|f-f_{k}\right|<\delta$.
Let $\mu_{k}$ be a covering measure of $f_{k}$ by $g$ with $\int hd\mu_{k}<N\left(f,g,h\right)-\eps/2$,
and let $\lambda$ be the Lebesgue measure on $\R^{n}$. Then 
\[
\left(\mu_{k}+c\cdot\lambda\right)*g\ge f_{k}+\delta\ge f
\]
which means that $\mu_{k}+c\cdot\lambda$ is a covering measure of
$f$ by $g$. However, we then have that 
\[
\int hd\mu_{k}+c\int hdx<N\left(f,g,h\right),
\]
a contradiction. 
\end{proof}
\begin{proof}
[Proof of Theorem \ref{thm:SD_ext_1}] First note that $N\left(f,g_{k}\right)$
is a bounded sequence as clearly $N\left(f,g_{k}\right)\le N\left(f,g\right)$
for each $k$. Moreover, $N\left(f,g_{k}\right)$ is also clearly
non-decreasing, and thus converges to some limit. Let $\left(\mu_{k}\right)$
be a sequence of covering measures of $f$ by $g_{k}$ (in $\mm$)
such that 
\[
N\left(f,g_{k}\right)\le\mu_{k}\left(\R^{n}\right)\le N\left(f,g_{k}\right)+1/k.
\]
The Banach-Alaoglu theorem tells us that we may assume without loss
of generality that $\left(\mu_{k}\right)$ converges in the weak{*}
topology to some non-negative measure $\mu\in\mm$. Clearly, we have
that $\left(\mu_{k}*g_{l}\right)\left(x\right)\ge f\left(x\right)$
for all $l\le k$. Fixing $l$ and taking the limit $k\to\infty$
implies that $\left(\mu*g_{l}\right)\left(x\right)\ge f\left(x\right)$.
By the monotone convergence theorem, we may take the limit $l\to\infty$
and get that $\left(\mu*g\right)\left(x\right)\ge f\left(x\right)$.
Since $x$ is arbitrary, it follows that $\mu$ is a covering measure
of $f$ by $g$. The fact that $\mu\left(\R^{n}\right)\le\liminf\mu_{k}\left(\R^{n}\right)$
follows from the fact that a norm is lower semi-continuous with respect
to weak{*} convergence. Therefore we have that 
\[
\mu\left(\R^{n}\right)\le\liminf\mu_{k}\left(\R^{n}\right)=\lim N\left(f,g_{k}\right)\le N\left(f,g\right).
\]
Since $\mu$ is a covering measure of $f$ by $g$, it follows that
$\mu\left(\R^{n}\right)=N\left(f,g\right)=\lim N\left(f,g_{k}\right)$.

To show that $M\left(f,\gr\right)=N\left(f,g\right)$, recall first
that, by Proposition \ref{prop:WeakDuality}, $M\left(f,\gr\right)\le N\left(f,g\right)$.
On the other hand, by Theorem \ref{thm:StrongDuality} and the above,
we have that $N\left(f,g\right)=\lim N\left(f,g_{k}\right)=\lim M\left(f,\gr\right)\le M\left(f,\gr\right)$.
Thus, the equality $M\left(f,\gr\right)=N\left(f,g\right)$ holds. 
\end{proof}
Finally, we prove Theorem \ref{thm:SD_cov_gen}: 
\begin{proof}
[Proof of Theorem \ref{thm:SD_cov_gen}.]Let $\gamma=\inf\left\{ \alpha\,:\,(f,\alpha)\in A(K)\right\} $.
Since $(f,\alpha)\in A(K)$ if and only if there exists a covering
measure $\mu$ of $f$ by $g$ with $\int h\,d\mu=\alpha$, we see
that $\gamma=N(f,g,h).$ Since $A\left(K\right)$ is closed and non-empty,
there exists an optimal covering measure $\mu_{0}\in\mm^{+}$ of $f$
by $g$ such that $\mu_{0}*g\ge f$ and $\gamma=\int h\,d\mu_{0}$.

Let $\beta:=M\left(f,\gr,h\right)$. By Proposition \ref{prop:WeakDuality},
we have that $\beta\le\gamma$. Let $\eps>0$. We will next prove
that there is a $\gr$-separated measure $\rho$ such that $\iprod{\rho}f\ge\gamma-\eps$.
This would imply that $\gamma\le\beta+\eps$ and therefore $\gamma=\beta$.
Since $A\left(K\right)$ is closed and convex, the Hahn-Banach separation
theorem implies that the point $\left(f,\gamma-\eps\right)$ can be
strictly separated from $A\left(K\right)$. In other words, there
exists a pair $\left(\rho,\sigma\right)\in\mm\oplus\R$ and a number
$\alpha$ such that 
\begin{equation}
\iprod{\rho}f+\sigma\left(\gamma-\eps\right)>\alpha\label{eq:sep1-3}
\end{equation}
and 
\begin{equation}
\iprod{\rho}{\mu*g-\varphi}+\sigma\int h\,d\mu<\alpha\label{eq:sep2-3}
\end{equation}
for all $\left(\mu,\varphi\right)\in K.$ Choosing $\left(\mu,\varphi\right)=\left(0,0\right)$
implies that $\alpha>0$. Suppose that for some $\left(\mu,\varphi\right)\in K$
we have $\iprod{\rho}{\mu*g-\varphi}+\sigma\int h\,d\mu>0.$ Since
$K$ is a cone, we may choose a sufficiently large $\lambda>0$ so
that inequality \eqref{eq:sep2-3} is violated for $\lambda\left(\mu,\varphi\right)\in K$.
Thus we must have that 
\begin{equation}
\iprod{\rho}f+\sigma\left(\gamma-\eps\right)>0\label{eq:sep1-1-2}
\end{equation}
and 
\begin{equation}
\iprod{\rho}{\mu*g-\varphi}+\sigma\int h\,d\mu\le0\label{eq:sep2-1-2}
\end{equation}
for all $\left(\mu,\varphi\right)\in K$. Define $\varphi_{0}=\mu_{0}*g-f$,
and observe that $\left(\mu_{0},\varphi_{0}\right)\in K$. Hence,
\[
\iprod{\rho}f+\sigma\gamma\le0.
\]
By subtracting the above inequality from \eqref{eq:sep1-1-2} we conclude
that $\sigma<0$ and, by scaling $\left(\rho,\sigma\right)$ if needed,
we can assume that $\sigma=-1$. Thus we have that 
\[
\iprod{\rho}f-\left(\gamma-\eps\right)>0
\]
and 
\[
\iprod{\rho}{\mu*g-\varphi}-\int h\,d\mu=\iprod{\mu}{\rho*\gr}-\iprod{\rho}{\varphi}-\int h\,d\mu\le0
\]
for all $\left(\mu,\varphi\right)\in K.$ In particular, for $\left(0,\varphi\right)\in K$
we get that $\iprod{\rho}{\varphi}\ge0$ for all $\varphi\ge0$ which
means that $\rho\ge0$, and for $\left(\mu,0\right)\in K$ we get
that $\int\left(\rho*\gr\right)d\mu-\int h\,d\mu\le0$ for every $\mu\in\mm^{+}$,
which means that $\rho*\gr\le h$. Therefore $\rho$ is a $\gr$-separated
measure with $\iprod{\rho}f\ge\gamma-\eps$, as desired.
\end{proof}

\section{Volume bounds\label{sec:Volume-bounds}}

As with classical covering numbers, the simple but strong tool of
volume bounds plays a significant role in the the theory and in the
proofs. In this section we provide several volume bounds for functional
covering numbers, bounds which we then apply in the next sections. 

When dealing with the weight function $h\equiv1$, we denote for short
$N\left(f,g\right)=N^{1}\left(f,g\right)$.

We shall mainly be concerned with log-concave functions. A function
$f:\R^{n}\to\R^{+}$ is said to be log-concave if $f$ is upper semi-continuous
and $\log f$ is concave. In addition, $f$ is said to be a geometric
log-concave function if it is log-concave and $\max f=f\left(0\right)=1$.
We will mainly consider geometric log-concave functions with finite
and positive integral and denote the class of all such functions by
$\lcg$. The class of log-concave functions is considered to be the
usual generalization of convex bodies in Asymptotic Geometric Analysis.
Numerous objects, notions, inequalities and constructions have been
extended from convex geometry to the realm of log-concave functions.
This provides a rich theory, and many times the resulting theorems
can be applied into convexity again to gain new insight and stronger
results. For an extensive description of these ideas and the state
of the art see \cite{Milman08} and \cite{AGM15}.

Classical covering and separation numbers admit simple bounds in terms
of the volumes of the bodies involved. One has (see e.g. \cite[Chapter 4]{AGM15})
\[
\frac{\vol(K)}{\vol(T)}\le N(K,T)\le\frac{\vol(2K-T)}{\vol(T)}.
\]
These bounds, while very simple to prove, are extremely useful and
in many cases suffice for covering numbers estimates to provide tight
results.

In this section we prove some analogous bounds, in which the integral
of a function plays the role of volume. The role of Minkowski addition
is played the by sup-convolution of two functions $f,g:\R^{n}\to\R$,
which we recall is
\[
\left(f\star g\right)\left(x\right)=\sup_{y}f\left(y\right)g\left(x-y\right).
\]
As mentioned above, this convolution plays an important role in the
geometry of log-concave function as a natural extension of the Minkowski
sum of convex bodies (where indeed $\one_{K}\star\one_{T}=\one_{K+T}$
for two convex bodies $K,T\sub\R^{n}$, where $\mathbbm1_{A}$ denotes
the indicator function of a set $A$). For example, under these analogies
one may interpret the Prékopa-Leindler inequality as an extension
of Brunn-Minkowski inequality, see e.g., \cite{KM05}. We prove
\begin{thm}
\label{thm:Vol_general} Let $f,g\in\lcg$. Then for every $p>1$
we have 
\[
\frac{\int f\left(x\right)dx}{\int g\left(x\right)dx}\le N\left(f,g\right)\le\frac{\int\left(f\star g_{-}^{p-1}\right)\left(x\right)dx}{\int g_{-}^{p}\left(x\right)dx}.
\]
\end{thm}
We remark that the left hand side inequality actually holds for any
two functions $f$ and $g$, whereas the right hand side inequality
is in general an upper bound for $M\left(f,\gr\right)$, which in
the log-concave case is equal to $N\left(f,g\right)$, a fact which
follows from approximation arguments (see Theorem \ref{thm:Sd-lcf}).
In any setting in which strong duality between covering and separation
holds (such as geometric log-concave functions) the above bounds hold
precisely as stated in Theorem \ref{thm:Vol_general}.
\begin{proof}
[Proof of Theorem \ref{thm:Vol_general}] Let $\mu$ be a covering
measure of $f$ by $g$. Then 
\[
\int f\left(x\right)dx\le\int\left(\mu*g\right)\left(x\right)dx=\mu\left(\R^{n}\right)\int g\left(x\right)dx.
\]
Since $\mu$ is an arbitrary covering measure, we conclude that $N\left(f,g\right)\ge\frac{\int f}{\int g}$.

Next, let $\rho$ be a $g$-separated measure and let $p>1$. Then
\begin{align*}
\int fd\rho\int g^{p}\left(x\right)dx & =\int\int f\left(y\right)g^{p}\left(x-y\right)dx\,d\rho\left(y\right)\le\int\int\sup_{z}\{f(z)g^{p-1}(x-z)\}g(x-y)dxd\rho\left(y\right)\\
 & =\int\int\left(f\star g^{p-1}\right)\left(x\right)g(x-y)dxd\rho\left(y\right)=\int\left(f\star g^{p-1}\right)\left(x\right)\left(\int g\left(x-y\right)d\rho\left(y\right)\right)dx\\
 & \le\int\left(f\star g^{p-1}\right)\left(x\right)dx.
\end{align*}
Since $\rho$ is an arbitrary $g$-separated measure, it follows that
$M\left(f,g\right)\le\frac{\int\left(f\star g^{p-1}\right)\left(x\right)dx}{\int g^{p}\left(x\right)dx}$.
As $f,g\in\lcg$, Theorem \ref{thm:Sd-lcf} tells us  that $N\left(f,g\right)=M\left(f,\gr\right)$,
which completes the proof.
\end{proof}
\begin{rem}
\noindent The above volume bounds can be written for the general weighted
covering number $N(f,g,h)$ as follows: for the right hand side, simply
use that $\int g\left(x-y\right)d\rho\left(y\right)dx\le h(x)$ to
get 
\[
M\left(f,g,h\right)\le\frac{\int\left(f\star g^{p-1}\right)\left(x\right)h(x)dx}{\int g^{p}\left(x\right)dx}.
\]
The left hand side can be generalized for weight functions $h$ satisfying
$h(x+y)\le h(x)h(y)$ as follows 
\begin{eqnarray*}
\int f\left(x\right)h(x)dx & \le & \int\left(\mu*g\right)\left(x\right)h(x)dx\\
 & \le & \int\int g(x-t)h(x-t)h(t)d\mu(t)dx\\
 & = & \int h(x)d\mu(x)\int g(x)h\left(x\right)dx.
\end{eqnarray*}
We get 
\[
\frac{\int f\left(x\right)h(x)dx}{\int g\left(x\right)h(x)dx}\le N\left(f,g,h\right).
\]
\end{rem}
We include one more pair of volume bounds which shall be useful for
us in further applications: 
\begin{thm}
\label{thm:Vol_even} Let $f,g\in\lcg$. Then 
\[
\frac{\int f^{2}\left(x\right)dx}{\|f*g_{-}\|_{\infty}}\le N\left(f,g\right)\le2^{n}\frac{\int f^{2}\left(x\right)dx}{\|f*g\|_{\infty}}.
\]
In the special case where $f,g$ are even functions we get 
\[
\frac{\int f^{2}\left(x\right)dx}{\int f\left(x\right)g\left(x\right)dx}\le N\left(f,g\right)\le2^{n}\frac{\int f^{2}\left(x\right)dx}{\int f\left(x\right)g\left(x\right)dx}.
\]
\end{thm}
The idea behind the proof of Theorem \ref{thm:Vol_even} is finding
specific covering and separating measures for a given pair of functions
$f,g$. The fact that such measures can be explicitly written is a
notable advantage in working with functional covering numbers over
classical covering numbers, where one usually cannot write down an
explicit covering for two given sets. This advantage was also exploited
in \cite{ArtSlom14-FracCov} where properties of an explicit covering
(uniform) measure played an important role in the proof of the fractional
Hadwiger conjecture. 
\begin{proof}
[Proof of Theorem \ref{thm:Vol_even}] We begin with the right hand
side. Assume that $\|f*g\|_{\infty}=f*g(x_{0})$. Consider the measure
$\mu$ with density $\frac{f^{2}\left(x/2+x_{0}/2\right)}{\|f*g\|_{\infty}}$
with respect to the Lebesgue measure. We claim that this is a covering
measure of $f$ by $g$. Indeed, since $f$ is log-concave we have
that $f^{2}\left(\frac{x-y+x_{0}}{2}\right)\ge f\left(x\right)f\left(-y+x_{0}\right)$
for all $x,y\in\R^{n}$. Therefore, it follows that 
\[
\left(\mu*g\right)\left(x\right)=\frac{\int f^{2}\left(\frac{x-y+x_{0}}{2}\right)g\left(y\right)dy}{f*g(x_{0})}\ge\frac{f\left(x\right)\int f\left(-y+x_{0}\right)g\left(y\right)dy}{f*g(x_{0})}=f\left(x\right).
\]
Thus 
\[
N\left(f,g\right)\le\frac{\int f^{2}\left(\frac{x-x_{0}}{2}\right)dx}{\|f*g\|_{\infty}}=2^{n}\frac{\int f^{2}}{\|f*g\|_{\infty}}.
\]
For the left hand side inequality, consider the measure $\rho$ with
density $\frac{f\left(x\right)}{\|f*g\|_{\infty}}dx$. Thus 
\[
\sup_{x}\left(\rho*g\right)\left(x\right)=\frac{\sup_{x}\int f\left(y\right)g\left(x-y\right)dy}{\|f*g\|_{\infty}}=1
\]
which means that $\rho$ is $g$-separated. Therefore 
\[
M\left(f,g\right)\ge\frac{\int f^{2}}{\|f*g\|_{\infty}}.
\]
By Theorem \ref{thm:Sd-lcf}, $M\left(f,g_{-}\right)=N\left(f,g\right)$
and the proof is complete.
\end{proof}
\begin{rem*}
\noindent An analogue with weight function $h$ for the right hand
side can be written. In the even case we get 
\[
N\left(f,g,h\right)\le2^{n}\frac{\int f^{2}\left(x\right)h(x)dx}{\int f\left(x\right)g\left(x\right)dx}.
\]
\end{rem*}

\section{A remark on the Functional Hadwiger conjecture\label{sec:Functional-Hadwiger-conjecture}}

A famous conjecture, known as the Levi-Hadwiger or the Gohberg-Markus
covering problem, was posed in \cite{Levi55}, \cite{Hadwiger57}
and \cite{GohMar60}. It states that in order to cover a convex body
by slightly smaller copies of itself, one needs at most $2^{n}$ copies.
More precisely: 
\begin{conjecture}
Let $K\sub\R^{n}$ be a convex body with non empty interior. Then
there exists $0<\lambda<1$ such that 
\[
N(K,\lambda K)\le2^{n}.
\]
Moreover, equality holds if and only if $K$ is a parallelotope.

\noindent An equivalent form of this conjecture is that $N(K,{\rm int}(K))\le2^{n}$,
where $\inte K$ is the interior of $K$.
\end{conjecture}
This problem has drawn much attention over the years, but not much
has been unraveled  so far. Our paper \cite{ArtSlom14-FracCov} has
addressed this problem by using fractional covering numbers for convex
sets. We showed, in the language of the current paper: 
\begin{thm}
\label{thm:WeightedHadwiger} Let $K\sub\R^{n}$ be a convex body.
Then 
\[
\lim_{\lambda\to1^{-}}N^{1}(1_{K},1_{\lambda K})\le\begin{cases}
2^{n} & K=-K\\
{2n \choose n} & K\neq-K
\end{cases}
\]
Moreover, for centrally symmetric $K$, ${\displaystyle \lim_{\lambda\to1^{-}}N_{\omega}(K,\lambda K)=2^{n}}$
if and only if $K$ is a parallelotope. 
\end{thm}
The bound $2^{n}$ for the non-centrally symmetric case remains a
conjecture even in the fractional setting. (Things get significantly
better if one considers $N(K,-{\rm int}(K))$, though.) In this note
we address the functional version. As we shall next demonstrate, it
follows from our volume bounds that 
\begin{thm}
\label{thm:FunctionalHadwiger}Let $f\in\lcg$ be an even geometric
log-concave function. Then 
\[
\lim_{\lambda\to1^{-}}N(f,f_{\lambda})\le2^{n}
\]
where $f_{\lambda}(x)=f(x/\lambda)^{\lambda}$. In the case where
$f$ is not even, we have that 
\[
\lim_{\lambda\to1^{-}}N\left(f,f_{\lambda}\right)\le4^{n}.
\]
\end{thm}
\begin{proof}
~Theorem \ref{thm:Vol_general}, applied with $p=2$, implies that 

\[
\lim_{\lambda\to1^{-}}N\left(f,f_{\lambda}\right)\le\lim_{\lambda\to1^{-}}\frac{\int\left(f\star f_{\lambda}\right)\left(x\right)dx}{\int f_{\lambda}^{2}\left(x\right)dx}=\frac{\int\left(f\star f\right)\left(x\right)dx}{\int f^{2}\left(x\right)dx}=\frac{\int f\,^{2}\left(x/2\right)dx}{\int f^{2}dx}=2^{n}.
\]
In the general case, where $f$ is not necessarily even, we use Fact
\ref{fact:sub_mult}, and Theorem \ref{thm:Vol_general} to obtain
\[
\lim_{\lambda\to1^{-}}N(f,f_{\lambda})\le\lim_{\lambda\to1^{-}}N\left(f,\left(f_{-}\right)_{\lambda}\right)N\left(\left(f_{-}\right)_{\lambda},f_{\lambda}\right)\le\lim_{\lambda\to1^{-}}\frac{\int f\star f_{\lambda}}{\int f_{\lambda}^{2}}\cdot\frac{\int f_{\lambda}\star f_{\lambda}}{\int f_{\lambda}^{2}}=\left(\frac{\int f\star f}{\int f^{2}}\right)^{2}=4^{n}.
\]
\end{proof}

\section{M-position for functions\label{sec:M-position-for}}

One of the deepest results in asymptotic geometric analysis is the
existence of an $M$-ellipsoid associated with a convex body, namely
that for every convex body $K\subset\R^{n}$ there exists an ellipsoid
of the same volume which can replace it, in many volume computations,
up to universal constants. This profound result was discovered by
V. Milman \cite{Milman86,Milman88}, and leads to many far reaching
conclusions, among them are the reverse Blaschke-Santaló inequality
\cite{BourgainMilman87} and the reverse Brunn-Minkowski inequality.
For a detailed account of this subject, see \cite[Ch. 8]{AGM15}.

In the functional setting, Klartag and Milman in \cite{KM05} showed
a reverse Brunn-Minkowski inequality for functions. This requires
a choice of a position of course. They proved the following functional
version of the inverse Brunn-Mikowski inequality. 
\begin{thm}
[Klartag-Milman \cite{KM05}] \label{thm:KM_inv-BM} For every $f:\R^{n}\to\left[0,\infty\right)$
which is an even geometric log-concave function, there exists $T_{f}\in SL_{n}$
such that the following holds: Let $f,h:\R^{n}\to\left[0,\infty\right)$
be even geometric log-concave functions. Then letting $\widetilde{f}=f\circ T_{f}$
and $\widetilde{h}=g\circ T_{h}$ one has 
\[
\left(\int\widetilde{f}\star\widetilde{h}\right)^{1/n}<C\left[\left(\int\widetilde{f}\right)^{1/n}+\left(\int\widetilde{h}\right)^{1/n}\right]
\]
where $C>0$ is a universal constant, independent of the dimension
and of $f$ and $h$. 
\end{thm}
In this section we extend their result. We show that two definitions
of $M$-position for functions are equivalent, one which uses volume
inequalities and another which uses functional covering numbers. We
then show that every geometric log-concave function admits an $M$-position
with a universal constant, a fact implicitly shown already in \cite{KM05}
but which now bears a stronger meaning as described in Theorem \ref{thm:M_pos_equiv}
below. While this fact can be deduced from Theorem \ref{thm:KM_inv-BM},
we give a independent proof which uses the covering number estimates
in $M$-positions of convex bodies. In particular we show that Theorem
\ref{thm:KM_inv-BM} may be extended to the non-even case, if the
center of mass of $f$ is assumed to be at the origin. The non-centrally-symmetric
case of the classical $M$-position was treated in \cite{PajorMilman2000}.

We would like to mention another result of a similar flavor, of Bobkov
and Madiman \cite{BoMa12}. They use another functional variation
for Minkowski addition coming from the addition of random variables
(thus pertaining to the usual convolution of the density functions)
and consider entropy instead of volume. Under this setting they show
the existence of positions for which a reverse Brunn-Minkowski-type
inequality holds. Their results are very different from ours, in particular
the sum of two indicators is no longer an indicator. However, their
results are in a very general setting of $\beta$-concave functions
(with the constants involved depending on the degree of concavity,
and becoming universal when the densities are log-concave). 

To simplify the exposition it is useful to first state and prove a
duality result in the flavor of König and Milman, and this is done
in Section \ref{subsec:Func_KM}. The two equivalent definitions for
$M$-position, and the proof that they are equivalent are given in
Section \ref{subsec:M_pos_equiv}, and the fact that Theorem \ref{thm:KM_inv-BM}
implies the existence of an $M$-position is given in Section \ref{subsec:M-pos_existence}.
Finally, in Section \ref{sec:direct_M_pos} we give a new proof of
the existence of $M$-position for functions. 

Throughout this section, $A\sim B$ means that $c^{n}B\le A\le BC^{n}$
for some universal constants $c,C>0$. Moreover, the value of such
universal constants may change from line to line, and the reader may
take the minimum (or maximum) of the constants appearing (which one
could name $C_{1},C_{2},$ etc.) as the final constant in the main
theorems.

\subsection{\label{subsec:Func_KM}A König-Milman type result for functions }

Let $\iprod{\cdot}{\cdot}$ denote the standard scalar product on
$\R^{n}$. The polar of a convex set $A\sub\R^{n}$ is defined by
$A^{\circ}=\left\{ v\in\R^{n}:\ \sup_{x\in K}\iprod vx\le1\right\} $.
Given a centrally symmetric convex body $A$ (i.e., a convex set with
non-empty interior), the polar set $A^{\circ}$ is again a centrally
symmetric convex body. The notion of duality is very basic in geometry
and analysis. It admit a natural functional extension which is the
Legendre transform for convex functions:
\[
\left(\L\varphi\right)\left(x\right)=\sup\iprod xy-\varphi\left(y\right).
\]
In the log-concave world this transform gives for a log-concave $f$
a natural dual $f^{*}=e^{-\L\varphi}$. This choice of duality has
been used in numerous works, for example in \cite{BallPhD,AKM2005,KM05}
where functional versions of the Santaló inequality and its reverse
were proven. We shall discuss below yet another candidate for the
``polar function'' of a geometric log-concave $f$.

Going back to geometric duality, a central question for covering numbers,
proposed by Piestch \cite{Piestch72} was to determine the relation
between $N(K,tT)$ and $N(T^{\circ},tK^{\circ})$, as functions of
$t\in\R^{+}$. This is called ``duality of entropy numbers''. Many
results on this question have been proven by now, see \cite{AGM15}.
One of them is the following well-known duality of entropy result
due to H. König and V. Milman \cite{KonigMilman87}: There exists
a numerical constant $C>0$ such that for any two centrally symmetric
convex bodies $K,T\sub\R^{n}$ one has 
\begin{equation}
C^{-n}N\left(T^{\circ},K^{\circ}\right)\le N\left(K,T\right)\le C^{n}N\left(T^{\circ},K^{\circ}\right).\label{eq:class_KoMi}
\end{equation}

Using the suitable corresponding notion of duality for log-concave
functions given by the Legendre transform, we prove an analogous functional
result: 
\begin{thm}
\label{thm:Func_Konig_Milman}There exists a numerical constant $C>0$
such that for any dimension $n$ and any two functions $f,g\in\lcg$
with center of mass at the origin, one has 
\[
C^{-n}N\left(g^{*},f^{*}\right)\le N\left(f,g\right)\le C^{n}N\left(g^{*},f^{*}\right).
\]
\end{thm}
\begin{proof}
Note that for every function $h\in\lcg$ one has 
\begin{equation}
\int h^{2}\le\int h\le2^{n}\int h^{2}.\label{eq:Func_KM_square}
\end{equation}
Indeed, the left hand side inequality follows from the fact that $h\le1$
and the right hand side inequality follows from $h\left(x\right)\le h^{2}\left(x/2\right)$
which holds due to the fact that $h$ is log-concave with $h\left(0\right)=1$. 

Assume first that $f$ and $g$ are both even functions. Using the
volume bound in Theorem \ref{thm:Vol_general} together with \eqref{eq:Func_KM_square}
we see that 
\[
N\left(g^{*},f^{*}\right)\le\frac{\int f^{*}\star g^{*}}{\int\left(f^{*}\right)^{2}}\le2^{n}\frac{\int f^{*}\star g^{*}}{\int f^{*}}.
\]
As $\left(f^{*}\star g^{*}\right)^{*}=fg$, the functional Santaló
inequality and its reverse (see \cite{AKM2005} and \cite{KM05})
 imply that 
\[
\frac{\int f^{*}\star g^{*}}{\int f^{*}}\le C^{n}\frac{\int f}{\int fg}
\]
for some absolute constant $C>0$. By Theorem \ref{thm:Vol_even}
and \eqref{eq:Func_KM_square}, we conclude that 
\[
N\left(g^{*},f^{*}\right)\le C_{1}^{n}\frac{\int f}{\int fg}\le C_{2}^{n}\frac{\int f^{2}}{\int fg}\le C_{2}^{n}N\left(f,g\right).
\]
To obtain the opposite inequality, one simply replaces the roles of
$f,g$ with $g^{*},f^{*}$, respectively.

Next, we prove the general case in which $f$ and $g$ are not necessarily
even. Firstly, by the above proof for even functions, we have that
\begin{equation}
\frac{1}{\left(4c\right)^{n}}N\left(g^{*}\star\refl g^{*},f^{*}\cdot\refl f^{*}\right)\le N\left(f\star\refl f,g\cdot\refl g\right)\le C_{2}^{n}N\left(g^{*}\star\refl g^{*},f^{*}\cdot\refl f^{*}\right).\label{eq:KM_Even}
\end{equation}
Secondly, by \cite[Theorem 2.2]{AloGonJimVil16}, which is a functional
version of the Rogers-Shephard inequality for the difference body,
for every geometric log-concave function $h\in\lcg$ with full-dimensional
support we have that 
\begin{equation}
\int h\star\refl h\le4^{n}\int h.\label{eq:RS_Alonso}
\end{equation}
Using the sub-multiplicativity (Fact \ref{fact:sub_mult}), monotonicity
of covering (Fact \ref{fact:cov_mono}) together with Theorem \ref{thm:Vol_even},
and \eqref{eq:RS_Alonso}, we have that 
\begin{align*}
N\left(f\star\refl f,g\cdot\refl g\right) & \le N\left(f\star\refl f,f\right)N\left(f,g\right)N\left(g,g\cdot\refl g\right)\le N\left(f\star\refl f,f\cdot\refl f\right)N\left(f,g\right)N\left(g\star\refl g,g\cdot\refl g\right)\\
 & \le4^{n}\frac{\int f\star f_{-}}{\int\left(f\star\refl f\right)\left(f\cdot\refl f\right)}N\left(f,g\right)\frac{\int g\star\gr}{\int\left(g\star\refl g\right)\left(g\cdot\refl g\right)}\\
 & \le4^{3n}\frac{\int f}{\int\left(f\cdot f_{-}\right)^{2}}N\left(f,g\right)\frac{\int g}{\int\left(g\cdot\refl g\right)^{2}}\\
 & \le4^{4n}\frac{\int f}{\int f\cdot\refl f}N\left(f,g\right)\frac{\int g}{\int g\cdot\refl g}.
\end{align*}
Hence, by Santaló inequality and its reverse, and \eqref{eq:RS_Alonso},
it follows that
\begin{align*}
N\left(f\star\refl f,g\cdot\refl g\right)\le & C_{1}^{n}\left(\int f\int f^{*}\star\refl f^{*}\right)N\left(f,g\right)\left(\int g\int g^{*}\star\refl g^{*}\right)\\
\le & C_{2}^{n}\left(\int f\int f^{*}\right)N\left(f,g\right)\left(\int g\int g^{*}\right)\\
\le & C_{3}^{n}N\left(f,g\right).
\end{align*}
 Similarly, we obtain that $N\left(f^{*}\star\refl f^{*},g^{*}\cdot\refl g^{*}\right)\le C_{3}^{n}N\left(g^{*},f^{*}\right).$
Together with \eqref{eq:KM_Even}, we conclude that 
\[
N\left(g^{*},f^{*}\right)\le N\left(g^{*}\star\refl g^{*},f^{*}\cdot\refl f^{*}\right)\le C_{1}^{n}N\left(f\star\refl f,g\cdot\refl g\right)\le C_{2}^{n}N\left(f,g\right),
\]
and 
\[
N\left(f,g\right)\le N\left(f\star\refl f,g\cdot\refl g\right)\le C_{1}^{n}N\left(g^{*}\star\refl g^{*},f^{*}\cdot\refl f^{*}\right)\le C_{2}^{n}N\left(g^{*},f^{*}\right),
\]
as claimed. 
\end{proof}
\begin{rem}
Since \cite{AKM2005} an extensive effort has been applied to determining
understanding the operation of ``duality'' for functions and investigating
$f^{*}$ together with other possible definitions. In \cite{AM2010,AM4}
it was shown that on the class $\lcg$ of geometric log-concave functions
there are precisely {\em two} order reversing bijections. One of
them is $e^{-\varphi}\mapsto e^{-\L\varphi}$, and the second, which
we shall denote by $f\mapsto f^{\circ}$, is less well known, and
is given by the formula $f=e^{-\varphi}\mapsto e^{-\A\varphi}$ where
\begin{equation}
\left(\A\varphi\right)(x)=\left\{ \begin{array}{ll}
\sup_{\{y\in\R^{n}:\varphi(y)>0\}}\frac{\langle x,y\rangle-1}{\varphi(y)} & \mbox{if}~~x\in\{\varphi(y)=0\}^{\circ}\\
+\infty & \mbox{if}~~x\not\in\{\varphi(y)=0\}^{\circ}
\end{array}\right.\ \label{newt}
\end{equation}
(with the convention $\sup\emptyset=0$.) For a detailed description
of this transform, geometric interpretations, properties and more,
see \cite{AM4}. It turns out that a result similar to Theorem \ref{thm:Func_Konig_Milman}
does not hold when replacing the Legendre-based duality with the polarity
transform. However, by slightly altering the polarity transform, one
can prove another Santaló-type inequality and its reverse, which leads
to a corresponding functional extension of Theorem \ref{thm:Func_Konig_Milman}.
These results will be stated in a precise form and proved in the forthcoming
\cite{Slo17}. 
\end{rem}

\subsection{\label{subsec:M_pos_equiv}The equivalence of the covering and volumetric
$M$-positions }

In this section we give two definitions for functional $M$-position
and show that they are equivalent. In particular, we will get that
in $M$-position we have a family of replacement-by-gaussians inequalities
which, we will see in the next section, are equivalent to Theorem
\ref{thm:KM_inv-BM}. To distinguish the two definitions, at least
until we show they are equivalent, we call the first volume-$M$-position
and the second covering-$M$-position. Denote $g_{0}(x)=\exp(-\frac{1}{2}|x|^{2})$,
so that $\int g_{0}=(2\pi)^{n/2}$. In general, for a positive definite
matrix $A$ let $g_{A}\left(x\right)=\exp(-\frac{1}{2}\iprod{Ax}x)$,
so that $\int g_{A}=\frac{(2\pi)^{n/2}}{\det^{1/2}A}$ and $g_{Id}=g_{0}$. 
\begin{defn}
\label{def:volMpos}Let $f\in\lcg$ and $C>0$. We say that $f$ is
in volume-$M$-position with constant $C$ if $\int f=(2\pi)^{n/2}$
and for every $h\in\lcg$, 
\[
\frac{1}{C^{n}}\int g_{0}\star h\le\int f\star h\le C^{n}\int g_{0}\star h
\]
 and 
\[
\frac{1}{C^{n}}\int g_{0}\star h\le\int f^{*}\star h\le C^{n}\int g_{0}\star h.
\]

For general $f$, if for $T_{f}\in GL_{n}$ we have that $\widetilde{f}=f\circ T_{f}$
is in volume-$M$-position with constant $C$, then we say that $g_{A}=g_{0}\circ T_{f}^{-1}$
is a volume-$M$-ellipsoid of $f$ with constant $C$. 
\end{defn}
\begin{rem}
\label{rem:M_pos_dual_vol}If $f$ is in volume-$M$-position, then
$f^{*}$ is not necessarily in volume-$M$-position as well, since
$\int f^{*}$ might not (and actually unless $f$ is gaussian, never
will) equal $(2\pi)^{n/2}$. However, the Blashcke-Santaló inequality
for functions \cite{AKM2005} states that for a log-concave function
with center of mass at the origin one has
\[
\int f\int f^{*}\le(2\pi)^{n}
\]
and thus if $f$ is centered and in volume-$M$-position then $\int f^{*}\le(2\pi)^{n/2}$.
Moreover, the reverse Blashcke-Santaló inequality for functions \cite{KM05}
states that there is a universal constant $C_{0}$ so that for a log-concave
function

\[
C_{0}^{n}\le\int f\int f^{*}
\]
and so we actually get that if $f$ is in volume-$M$-position with
constant $C$ then $f^{*}(cx)$ (for the normalizing $c$, which is
bounded between two universal constants) is in volume-$M$-position
with constant $C/c$. 
\end{rem}
Another possible definition for $M$-position is using covering numbers. 
\begin{defn}
\label{def:covMpos}Let $f\in\lcg$ and $C>0$. We say that $f$ is
in covering-$M$-position with constant $C$ if $\int f=(2\pi)^{n/2}$
and 
\[
N\left(f,g_{0}\right),N\left(g_{0},f\right),N\left(f^{*},g_{0}\right),N\left(g_{0},f^{*}\right)\le C^{n}.
\]
Again if there exists some $T_{f}\in GL_{n}$ such that $\widetilde{f}=f\circ T_{f}$
is in covering-$M$-position with constant $C$, we say that $g_{A}=g_{0}\circ T_{f}^{-1}$
is a covering-$M$-ellipsoid of $f$ with constant $C$. 
\end{defn}
\begin{rem}
\label{rem:M_pos_dual_cov}Again, if $f$ is in covering-$M$-position,
then $f^{*}$ is not necessarily in covering-$M$-position due to
normalization, but by the Blashcke-Santaló inequality for functions
and its reverse, $f^{*}(cx)$ (for the normalizing $c>0$) is in covering-$M$-position
with constant $2C/c$ since by our volume bounds,
\begin{align*}
N\left(f^{*}(cx),g_{0}(x)\right) & \le N\left(f^{*}(cx),f(x)\right)N\left(f(x),g_{0}(x)\right)\le\frac{\int f\left(cx\right)\star f\left(x\right)}{\int f^{2}\left(x\right)}C^{n}\\
 & \le\frac{\int f\left(cx\right)\star f\left(cx\right)}{\int f^{2}\left(x\right)}C^{n}=\frac{\int f^{2}\left(cx/2\right)}{\int f^{2}\left(x\right)}C^{n}=\left(\frac{2C}{c}\right)^{n}.
\end{align*}
\end{rem}
We claim that the two definitions coincide, up to a loss in the constants.
In other words
\begin{thm}
\label{thm:M_pos_equiv}For every $C>0$ there exists $C_{1}(C)>0$
such that if $f\in\lcg$ is in covering-$M$-position with constant
$C$ then it is in volume-$M$-position with constant $C_{1}$, and
if $f\in\lcg$ is in volume-$M$-position with constant $C$ then
it is in covering-$M$-position with constant $C_{1}$.
\end{thm}
\begin{proof}
Assume that $f$ is in volume-$M$-position with constant $C$. Then
by the volume inequality of Theorem \ref{thm:Vol_general} with $p=2$
we have that 
\[
N\left(f,g_{0}\right)\le\frac{\int f\star g_{0}}{\int g_{0}^{2}}\le C^{n}\frac{\int g_{0}\star g_{0}}{\int g_{0}^{2}}=\left(2C\right)^{n},
\]
and (using the inequality again, together with the fact that for geometric
log-concave functions $f^{2}\left(x/2\right)\ge f\left(x\right)$)
\begin{align*}
N\left(g_{0},f\right) & \le\frac{\int g_{0}\star f}{\int f^{2}}=2^{n}\frac{\int g_{0}\star f}{\int f^{2}\left(x/2\right)dx}\le\left(2C\right)^{n}\frac{\int g_{0}\star g_{0}}{\int f^{2}\left(x/2\right)dx}\\
 & =\left(4C\right)^{n}\frac{\int g_{0}^{2}\left(x\right)dx}{\int f^{2}\left(x/2\right)dx}\le\left(4C\right)^{n}\frac{\int g_{0}\left(x\right)dx}{\int f\left(x\right)dx}=\left(4C\right)^{n}
\end{align*}
The necessary bounds for $N\left(f^{*},g_{0}\right)$ and $N\left(g_{0},f^{*}\right)$
are obtained similarly, or by using Remark \ref{rem:M_pos_dual_vol}
which states that after normalization $f^{*}$ is in volume-$M$-position
too, and the normalizing constant is bounded by some $C_{1}^{n}$,
so it influences the estimates by at most some $C_{2}^{n}$.

We turn now to the other implication. Let $\mu$ be a covering measure
of $f$ by $g_{0}$, and $\nu$ a covering measure of $f^{*}$ by
$g_{0}$. Then using \eqref{eq:Sup-conv_Conv_ineq}, we have that
\[
\int f\star h\le\int\left(\mu*g_{0}\right)\star h\le\int\mu*\left(g_{0}\star h\right)=\mu\left(\R^{n}\right)\int g_{0}\star h\le C^{n}\int g_{0}\star h,
\]
and, similarly, 
\[
\int f^{*}\star h\le\int\nu*\left(g_{0}\star h\right)\le C^{n}\int g_{0}\star h.
\]
 As for the reverse inequality, let $\nu$ be a covering measure of
$g_{0}$ by $f$, and $\mu$ a covering measure of $g_{0}$ by $f^{*}$.
Then 
\[
\int g_{0}\star h\le\int\left(\nu*f\right)\star h\le\int\nu*\left(f\star h\right)\le C^{n}\int f\star h,
\]
and, similarly, 
\[
\int g_{0}\star h\le\int\mu*\left(f^{*}\star h\right)\le C^{n}\int f^{*}\star h.
\]
\end{proof}

\subsection{\label{subsec:M-pos_existence}Existence of functional $M$-position}

In this section we prove that every centered geometric log-concave
function admits an $M$-position with a universal $C$. We shall be
using Theorem \ref{thm:KM_inv-BM} which is for even functions. In
order to be able, in the proof, to take care of non-even functions
as well, we shall need the following lemma.
\begin{lem}
\label{lem:replace_by_even}Let $f\in\lcg$ with $\bary f=0$. Then
for every $g\in\lcg$, we have 
\[
N\left(f,g\right)\sim N\left(f\refl f,g\right)\sim N\left(f\star\refl f,g\right),
\]
and 
\[
N\left(g,f\right)\sim N\left(g,f\refl f\right)\sim N\left(g,f\star\refl f\right)
\]
\end{lem}
\begin{proof}
Firstly, since $f\refl f\le f\le f\star\refl f$, it follows that
$N\left(f\refl f,g\right)\le N\left(f,g\right)\le N\left(f\star\refl f,g\right)$.
Secondly, 
\begin{align}
N\left(f,ff_{-}\right) & \le\frac{\int f\star f\refl f}{\int\left(f\refl f\right)^{2}}\le2^{n}\frac{\int f\star\refl f}{\int f\refl f}\le8^{n}\frac{\int f}{c^{n}}\int f^{*}\star\refl{f^{*}}\label{eq:N_ffbar}\\
 & \le c_{1}^{n}\int f\int f^{*}\le c_{2}^{n},\nonumber 
\end{align}
and hence $N\left(f,g\right)\le N\left(f,f\refl f\right)N\left(f\refl f,g\right)\le c_{2}^{n}N\left(f\refl f,g\right).$
Thirdly, by Theorem \ref{thm:Func_Konig_Milman}, and  \eqref{eq:N_ffbar}
(where the roles of $f$ and $f^{*}$ are interchanged), it follows
that 
\begin{equation}
N\left(f\star\refl f,f\right)\le C^{n}N\left(f^{*},f^{*}\refl f^{*}\right)\le C_{1}^{n}\label{eq:N_fStar_fbar}
\end{equation}
and hence $N\left(f,g\right)\ge\frac{N\left(f\star\refl f,g\right)}{N\left(f\star\refl f,f\right)}\ge C_{1}^{-n}N\left(f\star\refl f,g\right)$. 

\noindent To conclude the above, we have $N\left(f,g\right)\sim N\left(f\refl f,g\right)\sim N\left(f\star\refl f,g\right)$.

Next, we show that $N\left(g,f\right)\sim N\left(g,ff_{-}\right)\sim N\left(g,f\star\refl f\right)$.
Note that 
\[
N\left(g,f\star\refl f\right)\le N\left(g,f\right)\le N\left(g,f\refl f\right)
\]
Moreover, by \eqref{eq:N_ffbar} and \eqref{eq:N_fStar_fbar}, it
follows that
\[
N\left(g,f\right)\le N\left(g,f\star\refl f\right)N\left(f\star\refl f,f\right)\le C^{n}N\left(g,f\star\refl f\right),
\]
and 
\[
N\left(g,f\right)\ge\frac{N\left(g,f\refl f\right)}{N\left(f,f\refl f\right)}\ge N\left(g,f\refl f\right)C^{-n}.
\]

The proof is thus complete.
\end{proof}
We next show how the reverse Brunn-Minkowski inequality for even functions
from \cite{KM05}, quoted as Theorem \ref{thm:KM_inv-BM} above, implies
the existence of an $M$-position for every geometric log-concave
function which is centered. 
\begin{prop}
\label{prop:existsMpos}There exists a universal constant $C>0$ such
that any centered geometric log-concave function $f$ admits a functional
$M$-position (as in Definition \ref{def:covMpos} or Definition \ref{def:volMpos}). 
\end{prop}
\begin{proof}
Assume first that $f$ is even. Let $g_{f}\left(x\right)=g_{0}\left(x/r\right)$
be the scaled standard gaussian such that $\int g_{f}=\int f$. By
the functional reverse Brunn-Minkowski inequality, there exists $T_{f}\in{\rm SL_{n}}$
such that for $\widetilde{f}=f\circ T_{f}$, 
\[
\int\widetilde{f}\star g_{f}\le C^{n}\left(\left(\int\widetilde{f}\right)^{1/n}+\left(\int g_{f}\right)^{1/n}\right)^{n}=\left(2C\right)^{n}\int g_{f}.
\]
Therefore, 
\[
N\left(\widetilde{f},g_{f}\right)\le\frac{\int\widetilde{f}\star g_{f}}{\int g_{f}^{2}}\le\left(2C\right)^{n}\frac{\int\widetilde{f}}{\int g_{f}^{2}}\le\left(4C\right)^{n}\frac{\int g_{f}}{\int g_{f}}=\left(4C\right)^{n}.
\]
Similarly, one shows that $N\left(g_{f},\widetilde{f}\right)\le\left(4C\right)^{n}$. 

Using Theorem \ref{thm:Func_Konig_Milman}, one similarly shows that
$N\left(\widetilde{f}^{*},g_{f}^{*}\right),N\left(g_{f}^{*},\widetilde{f}^{*}\right)\le C{}_{1}^{n}$.
By Fact \ref{fact:Cov_Under_Lin}, if $g_{0}=g_{f}\circ T$ then $\hat{f}=\widetilde{f}\circ T$
is in $M$-position.

Next, assume that $f$ is not even, but only centered at the origin.
By the first part of the proof, we can put the even function $f\cdot\refl f$
in $M$-position, which means that $N\left(f\refl f,g_{0}\right)\le C^{n}$,
and $N\left(g_{0},f\refl f\right)\le C^{n}$. By Lemma \ref{lem:replace_by_even},
on the one hand we have that 
\[
N\left(f,g_{0}\right)\sim N\left(f\refl f,g_{0}\right)\le C^{n},
\]
 and on the other hand, that 
\[
N\left(g_{0},f\right)\sim N\left(g_{0},f\star\refl f\right)\le N\left(g_{0},f\refl f\right)\le C^{n}.
\]

By Theorem \ref{thm:Func_Konig_Milman} we have 
\begin{align*}
N\left(f^{*},g_{0}\right) & \le N\left(g_{0},f\right)C^{n}\le C_{1}^{n},\\
N\left(g_{0},f^{*}\right) & \le N\left(f,g_{0}\right)C^{n}\le C_{1}^{n},
\end{align*}
which completes the proof.
\end{proof}
Finally, we show that if two functions are, up to normalization, in
functional $M$-position then they satisfy the functional reverse
Brunn-Minkowski inequality. 
\begin{prop}
\label{pro:impM} Suppose $f,h\in\lcg$ satisfy that $\tilde{f}\left(x\right)=f\left(x/z_{f}\right)$
and $\tilde{h}\left(x\right)=h\left(x/z_{h}\right)$ are in functional
$M$-position (where $z_{f}=\left((2\pi)^{n/2}/\int f\right)^{1/n}$
and $z_{h}=\left((2\pi)^{n/2}/\int h\right)^{1/n}$) with constant
$C>0.$ Then they satisfy
\[
\left(\int f\star h\right)^{1/n}\le C\left(\left(\int f\right)^{1/n}+\left(\int h\right)^{1/n}\right).
\]
\end{prop}
In particular, Proposition \ref{pro:impM} together with Theorem \ref{thm:MpositionMAIN}
imply Theorem \ref{thm:KM_inv-BM}.
\begin{proof}
First, for $r>0$, let $g_{r}$ denote the standard gaussian scaled
by a factor by $r$, namely $g_{r}\left(x\right)=g_{0}\left(x/r\right)$.
Note that $\int g_{r}=r^{n}\int g_{0}$. Moreover, a simple calculation
tells us that for any $r,s$, 
\[
\left(g_{r}\star g_{s}\right)=\left(g_{r}^{*}g_{s}^{*}\right)^{*}=\left(g_{\frac{1}{r}}g_{\frac{1}{s}}\right)^{*}=\left(g_{\frac{1}{\sqrt{r^{2}+s^{2}}}}\right)^{*}=g_{\sqrt{r^{2}+s^{2}}}.
\]
By assumption (using the definition of volume-$M$-position), we have
that for any $\psi\in\lcg$,
\[
\int\tilde{f}\star\psi\sim C_{1}^{n}\int g_{0}\star\psi,\,\,\,\text{and\,\,\,}\int\tilde{h}\star\psi\sim C_{2}^{n}\int g_{0}\star\psi.
\]
By replacing $\psi$ with $\tilde{\psi}\left(x\right)=\psi\left(x/z_{f}\right)$,
and using the facts that  $\left(\tilde{f}\star\tilde{\psi}\right)\left(x\right)=\left(f\star\psi\right)\left(x/z_{f}\right)$
and $\left(g_{0}\star\tilde{\psi}\right)\left(x\right)=\left(g_{1/z_{f}}\star\psi\right)\left(x/z_{f}\right)$,
it follows that $\int f\star\psi\sim C_{1}^{n}\int g_{1/z_{f}}\star\psi$
for any $\psi\in\lcg$. Similarly, we have that $\int h\star\psi\sim C_{2}^{n}\int g_{1/z_{h}}\star\psi$
for any $\psi\in\lcg$. Therefore, 

\begin{align*}
\int f\star h & \le C_{1}^{n}\int g_{1/z_{f}}\star h\le\left(C_{1}C_{2}\right)^{n}\int g_{1/z_{f}}\star g_{1/z_{h}}\\
 & =\left(C_{1}C_{2}\right)^{n}\int g_{\sqrt{z_{f}^{-2}+z_{h}^{-2}}}=\left(2C\right)^{n}\left(z_{f}^{-2}+z_{h}^{-2}\right)^{\frac{n}{2}}\int g_{0}\\
 & =\left(C_{1}C_{2}\right)^{n}\left(\left(z_{f}^{-n}\int g_{0}\right)^{2/n}+\left(z_{h}^{-n}\int g_{0}\right)^{2/n}\right)^{n/2}\\
 & =\left(C_{1}C_{2}\right)^{n}\left(\left(\int f\right)^{2/n}+\left(\int h\right)^{2/n}\right)^{n/2}\\
 & \le\left(C_{1}C_{2}\right)^{n}\left(\left(\int f\right)^{1/n}+\left(\int h\right)^{1/n}\right)^{n}.
\end{align*}
\end{proof}

\subsection{\label{sec:direct_M_pos}Direct proof for covering $M$-position}

In this section we give a direct proof of the existence of an $M$-position
for a centered log-concave geometric convex function, based on the
geometric theorem of Milman on the existence of an $M$-position for
convex bodies. One may restrict to the case where $f\in LC_{g}(\R^{n})$
is even, since the centered and not-necessarily even case will then
follow by the reasoning given in the previous section. Let $f\in LC_{g}(\R^{n})$
be even and define 
\[
K_{f}=\{x:f(x)>\exp(-n)\},
\]
which is a centrally symmetric convex body. This body was used in
\cite{KM05} as well.
\begin{lem}
\label{lem: Direct_M_1} Let $f\in LC_{g}(\R^{n})$, then we have
that 
\[
C^{-n}\int f\le\vol(K_{f})\le C^{n}\int f
\]
for some universal constant $C>1$.
\end{lem}
\begin{proof}
The smallest log-concave function with a given level set $K$ is $f_{0}=\exp(-\varphi_{0})$
where the epigraph of $\varphi_{0}$ is the convex hull of $\{0\}$
and the set $\{(x,r):x\in K,r\ge n\}\subset\R^{n+1}$. The integral
of this function is 
\begin{align*}
\int f_{0} & =\int_{0}^{\infty}e^{-t}\vol\{x:\varphi_{0}(x)\le t\}dt\\
 & =\int_{0}^{\infty}e^{-t}\left(\frac{t}{n}\right)^{n}\vol(K_{f_{0}})\,dt\ge\int_{n}^{\infty}e^{-t}\vol(K_{f_{0}})dt\ge\vol(K_{f_{0}})C^{-n}.
\end{align*}
for some universal $C>0$. Picking $K=K_{f}$, since $f_{0}\le f$
and they share the level set $K_{f}=K_{f_{0}}$, it follows that $\int f\ge C^{-n}\vol(K_{f})$
for any $f$. For the other direction, for any function $f=e^{-\varphi}$
we have that 
\begin{align*}
\int e^{-\varphi} & =\int_{0}^{\infty}e^{-t}\vol\{x:\varphi(x)\le t\}\,dt\\
 & \le\int_{0}^{n}e^{-t}\vol(K_{f})+\int_{n}^{\infty}e^{-t}\left(\frac{t}{n}\right)^{n}\vol(K_{f})\,dt\le\vol(K_{f})C^{n}
\end{align*}
for some universal $C>0$. 
\end{proof}
Recall that a convex body $K\subset\R^{n}$ is in $M$-position with
constant $C>0$ if it can be covered by $C^{n}$ copies of the Euclidean
ball $RB_{2}^{n}$ where $R=\left(\vol(K)/\vol(B_{2}^{n})\right)^{1/n}$.
(That is, $\vol(K)=\vol(RB_{2}^{n})$.) Under this condition, and
under the assumption that the center of mass of $K$ is at the origin,
it is well known that also $K^{\circ}$ is in $M$-position and that
$N(RB_{2}^{n},K)\le C_{1}^{n}$, $N(RK^{\circ},B_{2}^{n})\le C_{1}^{n}$
and $N(B_{2}^{n},RK^{\circ})\le C_{1}^{n}$ where $C_{1}$ depends
only on $C$. For these properties and more about $M$-position, see
\cite{AGM15}. 
\begin{lem}
\label{lem:Direct_M_2}Let $f\in LC_{g}(\R^{n})$ satisfy that $\int f=(2\pi)^{n/2}$
and that $K_{f}$ is in $M$-position. Then $N(f,g_{0})\le C^{n}$. 
\end{lem}
\begin{proof}
Since $K_{f}$ is in $M$-position, it can be covered by $C^{n}$
copies of $RB_{2}^{n}$ where $R$ satisfies $\vol\left(RB_{2}^{n}\right)=R^{n}\kappa_{n}=\vol(K_{f})$.
By Lemma \ref{lem: Direct_M_1}, $\vol(K_{f})$ is at most $(3C)^{n}$,
and hence $R\simeq\sqrt{n}$. We note that $f\le1_{K_{f}}+\exp(-n\|\cdot\|_{K_{f}})$
(the sum is no longer log-concave) and thus 
\begin{align}
N(f,g_{0}) & \le N(1_{K_{f}},g_{0})+N(\exp(-n\|\cdot\|_{K_{f}}),g_{0})\label{eq:cov_lvl}\\
 & \le C^{n}N(1_{RB_{2}^{n}},g_{0})+N(\exp(-n\|\cdot\|_{K_{f}}),g_{0})\nonumber 
\end{align}
To bound the first term, note that since $R\simeq\sqrt{n}$, 
\[
N(1_{RB_{2}^{n}},g_{0})\le\frac{2^{n}\vol(RB_{2}^{n})}{\int_{RB_{2}^{n}}g_{0}}\le2^{n}g_{0}^{-1}\left(R\right)\le C_{2}^{n}
\]
To bound the second term, note that for any fixed $K$, by covering
level sets of height $e^{-k}$ each time, we have that
\begin{eqnarray*}
N(\exp(-\|\cdot\|_{K}),1_{RB_{2}^{n}}) & \le & \sum_{k=0}^{\infty}e^{-k}N((k+1)K,RB_{2}^{n}).
\end{eqnarray*}
In the particular case where $K=nK_{f}$ we get that

\begin{eqnarray*}
N(\exp(-n\|\cdot\|_{K_{f}}),1_{\sqrt{n}B_{2}^{n}}) & \le & \sum_{k=0}^{\infty}e^{-k}N(\frac{(k+1)}{n}K_{f},\sqrt{n}B_{2}^{n})\\
 & \le & \sum_{k=0}^{\infty}e^{-k}N(K_{f},\sqrt{n}B_{2}^{n})N((k+1)B_{2}^{n},nB_{2}^{n})\\
 & \le & C^{n}\sum_{k=0}^{\infty}e^{-k}N((k+1)B_{2}^{n},nB_{2}^{n})\\
 & \le & C^{n}\sum_{k=0}^{\infty}e^{-k}\frac{\left(2(k+1)+n\right)^{n}}{n^{n}}\le C_{3}^{n}.
\end{eqnarray*}
Since $N(1_{\sqrt{n}B_{2}^{n}},g_{0})\le C_{2}^{n}$, it follows that
\[
N(\exp(-n\|\cdot\|_{K_{f}}),g_{0})\le N(\exp(-n\|\cdot\|_{K_{f}}),1_{\sqrt{n}B_{2}^{n}})N(1_{\sqrt{n}B_{2}^{n}},g_{0})\le C_{2}^{n}C_{3}^{n}.
\]
Putting these together into \eqref{eq:cov_lvl} we see that the proof
of the lemma is complete.
\end{proof}
\begin{lem}
\label{lem:Direct_M_3}Assume that $f\in LC_{g}(\R^{n})$ is even,
satisfies that $\int f=(2\pi)^{n/2}$ and that $K_{f}$ is in $M$-position.
Then $K_{f^{*}}$ is also in $M$-position and $\vol\left(K_{f^{*}}\right)\approx\vol\left(K_{f}\right)$. 
\end{lem}
\begin{proof}
Since $K_{f}$ is in $M$-position, it follows by the above remarks
that $K_{f}^{\circ}$ is also in $M$-position (here we use the assumption
that $f$ is even, so that $K_{f}$ must be centered). To see that
$K_{f^{*}}$ in $M$-position, we will use the fact that for any $s,t>0$,
\begin{equation}
t\{x:\varphi(x)\le t\}^{\circ}\sub\{y:\L\varphi(y)\le t\}\sub\left(t+s\right)\{x:\varphi(x)\le s\}^{\circ}.\label{eq:inc}
\end{equation}
For the proof of these inclusions, see \cite[Lemma 8]{FradMeyer2008}.
Let $R_{1}B_{2}^{n}$ be the Euclidean ball with volume $\vol(K_{f}^{\circ})$,
and $R_{2}B_{2}^{n}$ the Euclidean ball with volume $\vol(K_{f^{*}})$.
For $s=t=n$, the \eqref{eq:inc} reads $nK_{f}^{\circ}\sub K_{f^{*}}\sub2nK_{f^{*}}^{\circ}$,
from which it also follows that $nR_{1}\le R_{2}\le2nR_{1}$. Combining
the above, we have that $K_{f^{*}}$ is in $M$-position. Indeed,
\[
N\left(K_{f^{*}},R_{2}B_{2}^{n}\right)\le N\left(2nK_{f}^{\text{\ensuremath{\circ}}},nR_{1}B_{2}^{n}\right)\le C^{n},
\]
and
\[
N\left(R_{2}B_{2}^{n},K_{f^{*}}\right)\le N\left(2nR_{1}B_{2}^{n},nK_{f}^{\circ}\right)\le C^{n}.
\]
The two remaining inequalities are immediately implied by \eqref{eq:class_KoMi}.

It remains to show that if $\vol(K_{f})\approx1$ then so is $\vol(K_{f^{*}})$.
Having the inclusions in \eqref{eq:inc}, the reverse Santaló inequality,
and Lemma \ref{lem: Direct_M_1}, we see that
\[
\vol(K_{f^{*}})=\vol\{y:\L\varphi(y)\le n\}\ge n^{n}\vol\{x:\varphi(x)\le n\}^{\circ}\ge a_{1}^{n}n^{n}\kappa_{n}^{2}/\vol(K_{f})\ge a^{n}.
\]
By changing the roles of $f$ and $f^{*}$, we get $\vol\left(K_{f^{*}}\right)\approx1$,
as required. 
\end{proof}
Combining Lemmas \ref{lem:Direct_M_2} and \ref{lem:Direct_M_3},
and Theorem \ref{thm:Func_Konig_Milman}, we get 
\begin{lem}
\label{lem: Direct_M_4}Let $f\in LC_{g}(\R^{n})$ be an even function,
which satisfies that $\int f=(2\pi)^{n/2}$ and that $K_{f}$ is in
$M$-position. Then 
\[
\max\{N(f,g_{0}),N(f^{*},g_{0}),N(g_{0},f),N(g_{0},f^{*})\}\le C^{n}
\]
\end{lem}
\begin{proof}
By Lemma \ref{lem: Direct_M_1} we know that $\vol(K_{f})\simeq1$
and by Lemma \ref{lem:Direct_M_3} also $\vol(K_{f^{*}})\simeq1$
and both these bodies are in $M$-position. By Lemma \ref{lem:Direct_M_2}
this implies $\max\{N(f,g_{0}),N(f^{*},g_{0})\}\le C^{n}$ for some
universal $C$. Using Theorem \ref{thm:Func_Konig_Milman} we get
the other two inequalities (possibly altering the value of $C$, but
keeping it universal nevertheless). 
\end{proof}
We have seen a direct proof of the covering numbers estimates in Theorem
\ref{thm:MpositionMAIN} for the even geometric log-concave case.
Indeed, the mapping $T_{f}\in GL_{n}(\R)$ is simply chosen so that
$T_{f}^{-1}K_{f}=K_{\tilde{f}}$ is in $M$-position and such that
$\int\tilde{f}=(2\pi)^{n/2}$, that is, $\det T_{f}=\int f/(2\pi)^{n/2}$.
To prove the covering numbers bound for the case of centered but not
necessarily even functions we follow the exact same reasoning as in
the proof of Proposition \ref{prop:existsMpos}. We put the even function
$f\refl f$ in $M$-position, use Lemma \ref{lem:replace_by_even},
which implies that on the one hand
\[
N\left(f,g_{0}\right)\sim N\left(f\refl f,g_{0}\right)\le C^{n},
\]
and on the other hand, 
\[
N\left(g_{0},f\right)\sim N\left(g_{0},f\star\refl f\right)\le N\left(g_{0},f\refl f\right)\le C^{n}.
\]
Finally, using that the function $f$ is centered, as well as $g_{0}$,
Theorem \ref{thm:Func_Konig_Milman} implies that also 
\begin{align*}
N\left(f^{*},g_{0}\right) & \le N\left(g_{0},f\right)C^{n}\le C^{2n},\\
N\left(g_{0},f^{*}\right) & \le N\left(f,g_{0}\right)C^{n}\le C^{2n},
\end{align*}
which completes the proof of the covering numbers bound for a general
centered function. By Theorem \ref{thm:M_pos_equiv} we get that the
other estimates in Theorem \ref{thm:MpositionMAIN} hold as well.
\begin{rem}
It would be interesting to give an independent proof for the existence
of functional $M$-positions of log-concave functions, without using
the analogue classical statement for convex bodies, and perhaps even
to show that there are $\alpha$-regular $M$-positions of functions
in the sense of Pisier (see e.g., \cite{PisierBook89}, and \cite{AGM15}). 
\end{rem}
\bibliographystyle{amsplain}
\addcontentsline{toc}{section}{\refname}\bibliography{MATH_April_2_2017}

\end{document}